\documentclass{amsart}
\usepackage{amssymb}

\usepackage{amscd}
\input{tcilatex}
\input tcilatex

\begin{document}
\author{Marek Kara\'{s}}
\title{There is no tame automorphism of $\Bbb{C}^{3}$ with multidegree $(4,5,6)$}
\keywords{polynomial automorphism, tame automorphism, multidegree.\\
\textit{2010 Mathematics Subject Classification:} 14Rxx,14R10}
\date{}
\maketitle

\begin{abstract}
It is known that not each triple $\left( d_{1},d_{2},d_{3}\right) $ of
positive integers is a multidegree of a tame automorphism of $\Bbb{C}^{3}.$
In this paper we show that there is no tame automorphism of $\Bbb{C}^{3}$
with multidegree $\left( 4,5,6\right) .$ To do this we show that there is no
pair of polynomials $P,Q\in \Bbb{C}\left[ x,y,z\right] $ with certain
properties. These properties do not seem particularly restrictive so the
non-existence result can be interesting in its own right.
\end{abstract}

\section{Introduction}

Let $F=\left( F_{1},\ldots ,F_{n}\right) :\Bbb{C}^{n}\rightarrow \Bbb{C}^{n}$
be any polynomial mapping. By multidegree of $F,$ denoted $\limfunc{mdeg}F,$
we mean the sequence $\left( \deg F_{1},\ldots ,\deg F_{n}\right) .$ We can
also consider the map $\limfunc{mdeg}:\limfunc{End}\left( \Bbb{C}^{n}\right)
\rightarrow \Bbb{N}^{n},$ where $\limfunc{End}\left( \Bbb{C}^{n}\right) $
denotes the set of polynomial endomorphisms of $\Bbb{C}^{n}.$ It is trivial
that $\limfunc{mdeg}\left( \limfunc{Aut}\left( \Bbb{C}^{1}\right) \right)
=\left\{ 1\right\} ,$ where $\limfunc{Aut}\left( \Bbb{C}^{n}\right) $
denotes the group of polynomial automorphisms of $\Bbb{C}^{n}.$ Let $%
\limfunc{Tame}\left( \Bbb{C}^{n}\right) $ denote the group of tame
automorphisms of $\Bbb{C}^{n}.$ Then, by the Jung \cite{Jung} and van der
Kulk \cite{Kulk} theorem, we have $\limfunc{mdeg}\left( \limfunc{Aut}\left( 
\Bbb{C}^{2}\right) \right) =\limfunc{mdeg}\left( \limfunc{Tame}\left( \Bbb{C}%
^{2}\right) \right) =\left\{ \left( d_{1},d_{2}\right) \in \Bbb{N}%
^{2}:d_{1}|d_{2}\text{ or }d_{2}|d_{1}\right\} .$

In higher dimensions the propblem is still not well recognized, even in
dimension three. The very first result in this direction \cite{Karas} says
that there is no tame automorphism of $\Bbb{C}^{3}$ with multidegree $\left(
3,4,5\right) .$ In the same paper it was also observed that for all $%
d_{3}\geq d_{2}\geq 2,$ $\left( 2,d_{2},d_{3}\right) \in \limfunc{mdeg}%
\left( \limfunc{Tame}\left( \Bbb{C}^{3}\right) \right) .$ Later \cite{Karas2}
it was proven that if $p_{2}>p_{1}>2$ are prime numbers, then $\left(
p_{1},p_{2},d_{3}\right) \in \limfunc{mdeg}\left( \limfunc{Tame}\left( \Bbb{C%
}^{3}\right) \right) \,$if and only if $d_{3}\in p_{1}\Bbb{N}+p_{2}\Bbb{N}$
(i.e. $d_{3}\,$is a linear combination of $p_{1}$ and $p_{2}$ with
coefficients in $\Bbb{N}).$ The next step was establishing \cite{Karas3} the
equality $\left\{ \left( 3,d_{2},d_{3}\right) :3\leq d_{2}\leq d_{3}\right\}
\cap \limfunc{mdeg}\left( \limfunc{Tame}\left( \Bbb{C}^{3}\right) \right)
=\left\{ \left( 3,d_{2},d_{3}\right) :3\leq d_{2}\leq d_{3},\text{ }3|d_{2}%
\text{ or }d_{3}\in 3\Bbb{N}+d_{2}\Bbb{N}\right\} .$ A similar equality can
be shown for triples of the form $\left( 5,d_{2},d_{3}\right) $ with $5\leq
d_{2}\leq d_{3}.\,$But this is more difficult to prove \cite{Karas4}. In
establishing a similar equality for the set $\left\{ \left(
7,d_{2},d_{3}\right) :7\leq d_{2}\leq d_{3}\right\} \cap \limfunc{mdeg}%
\left( \limfunc{Tame}\left( \Bbb{C}^{3}\right) \right) $ there is one
obstruction: the triple $\left( 7,8,12\right) .$

Notice that $3,5$ and $7$ are prime numbers. For $d_{1}=4,$ the first
composite number, we have no description of the entire set $\left\{ \left(
4,d_{2},d_{3}\right) :4\leq d_{2}\leq d_{3}\right\} \cap \limfunc{mdeg}%
\left( \limfunc{Tame}\left( \Bbb{C}^{3}\right) \right) ,$ but only some
partial information. It is not hard to prove that if $d_{3}\geq d_{2}\geq 4$
are even numbers, then $\left( 4,d_{2},d_{3}\right) \in \limfunc{mdeg}\left( 
\limfunc{Tame}\left( \Bbb{C}^{3}\right) \right) .\,$Also we can prove that
if $d_{3}\geq d_{2}\geq 4$ are odd, then $\left( 4,d_{2},d_{3}\right) \in 
\limfunc{mdeg}\left( \limfunc{Tame}\left( \Bbb{C}^{3}\right) \right) $ if
and only if $d_{3}\in 4\Bbb{N}+d_{2}\Bbb{N}.$ But if $d_{2}$ is odd and $%
d_{3}$ is even, or vice versa, we still do not have a complete description
of the set $\left\{ \left( 4,d_{2},d_{3}\right) :4\leq d_{2}\leq
d_{3}\right\} \cap \limfunc{mdeg}\left( \limfunc{Tame}\left( \Bbb{C}%
^{3}\right) \right) .$ The first unknown (up to now) thing has been whether $%
\left( 4,5,6\right) \in \limfunc{mdeg}\left( \limfunc{Tame}\left( \Bbb{C}%
^{3}\right) \right) .$

The goal of this paper is to prove the following theorem.

\begin{theorem}
\label{tw_main_4910}There is no tame automorphism of $\Bbb{C}^{3}$ with
multidegree $\left( 4,5,6\right) .$
\end{theorem}

To do this, we first prove that existence of such a tame automorphism would
imply existence of pair of polynomials $P,Q\in \Bbb{C}\left[ x,y,z\right] $
with certain properties. Next we show that such a pair does not exist. This
is the most difficult step. Since these properties do not look very
restrictive, the non-existence result can be interesting in its own right.
To prove it we develop a method that can be called $H$-reduction.

\section{The first reduction}

In this section we prove that existence of a tame automorphism of $\Bbb{C}%
^{3}$ with multidegree $\left( 4,5,6\right) $ implies existence of a pair of
polynomials with some special properties.

\begin{theorem}
\label{tw_456_implies}If there exists a tame automorphism $F=\left(
F_{1},F_{2},F_{3}\right) $ of $\Bbb{C}^{3}$ with $\limfunc{mdeg}F=\left(
4,5,6\right) ,$ then there exists a pair of polynomials $P,Q\in \Bbb{C}%
\left[ x,y,z\right] $ such that 
\begin{eqnarray*}
P &=&x+P_{2}+P_{3}+P_{4},\qquad P_{4}\neq 0, \\
Q &=&z+Q_{2}+\cdots +Q_{6},\qquad Q_{6}\neq 0,
\end{eqnarray*}
and 
\begin{equation*}
\deg \left[ P,Q\right] \leq 3,
\end{equation*}
where $P_{i},Q_{i}$ are homogeneous polynomials of degree $i.$
\end{theorem}

Let us recall that for any $f,g\in k[x_{1},\ldots ,x_{n}]$ we denote by $%
\left[ f,g\right] $ the Poisson bracket of $f$ and $g,$ which is the
following formal sum: 
\begin{equation*}
\sum_{1\leq i<j\leq n}\left( \frac{\partial f}{\partial x_{i}}\frac{\partial
g}{\partial x_{j}}-\frac{\partial f}{\partial x_{j}}\frac{\partial g}{%
\partial x_{i}}\right) \left[ x_{i},x_{j}\right] ,
\end{equation*}
where $[x_{i},x_{j}]$ are formal objects satisfying the condition 
\begin{equation*}
\lbrack x_{i},x_{j}]=-[x_{j},x_{i}]\qquad \text{for all }i,j.
\end{equation*}
We also define 
\begin{equation*}
\deg \left[ x_{i},x_{j}\right] =2\qquad \text{for all }i\neq j,
\end{equation*}
$\deg 0=-\infty $ and 
\begin{equation*}
\deg \left[ f,g\right] =\max_{1\leq i<j\leq n}\deg \left\{ \left( \frac{%
\partial f}{\partial x_{i}}\frac{\partial g}{\partial x_{j}}-\frac{\partial f%
}{\partial x_{j}}\frac{\partial g}{\partial x_{i}}\right) \left[
x_{i},x_{j}\right] \right\} .
\end{equation*}

Since $2-\infty =-\infty ,$ we have 
\begin{equation*}
\deg [f,g]=2+\underset{1\leq i<j\leq n}{\max }\deg \left( \frac{\partial f}{%
\partial x_{i}}\frac{\partial g}{\partial x_{j}}-\frac{\partial f}{\partial
x_{j}}\frac{\partial g}{\partial x_{i}}\right) .
\end{equation*}

From the above equality we have 
\begin{equation}
\deg \left[ f,g\right] \leq \deg f+\deg g.
\label{deg_Poisson_bracket_deg_f_g}
\end{equation}

Actually, we can show more:

\begin{theorem}
\label{tw_456_implies_rown}If there exists a tame automorphism $F=\left(
F_{1},F_{2},F_{3}\right) $ of $\Bbb{C}^{3}$ with $\limfunc{mdeg}F=\left(
4,5,6\right) ,$ then there exists a pair of polynomials $P,Q\in \Bbb{C}%
\left[ x,y,z\right] $ such that 
\begin{eqnarray*}
P &=&x+P_{2}+P_{3}+P_{4},\qquad P_{4}\neq 0, \\
Q &=&z+Q_{2}+\cdots +Q_{6},\qquad Q_{6}\neq 0,
\end{eqnarray*}
and 
\begin{equation*}
\deg \left[ P,Q\right] =3,
\end{equation*}
where $P_{i},Q_{i}$ are homogeneous polynomials of degree $i.$
\end{theorem}

\subsection{Some useful results}

Here we collect some useful results from other papers. The first one is the
following result of Shestakov and Umirbaev.

\begin{theorem}
\textit{(\cite{sh umb1}, Theorem 2)}\label{tw_deg_g_fg} Let $f,g\in \Bbb{C}%
[x_{1},\ldots ,x_{n}]$ be a $p$-reduced pair, and let $G(x,y)\in \Bbb{C}%
[x,y] $ with $\deg _{y}G(x,y)=pq+r,0\leq r<p.$ Then 
\begin{equation*}
\deg G(f,g)\geq q\left( p\deg g-\deg g-\deg f+\deg [f,g]\right) +r\deg g.
\end{equation*}
\end{theorem}

The notions of a *-reduced and a $p$-reduced pair, used in the above
theorem, are defined as follows.

\begin{definition}
\textit{(\cite{sh umb1}, Definition 1) }\label{def_*-red}A pair $f,g\in \Bbb{%
C}[x_{1},\ldots ,x_{n}]$ is called *-reduced if\newline
(i) $f,g$ are algebraically independent;\newline
(ii) $\overline{f},\overline{g}$ are algebraically dependent, where $%
\overline{h}$ denotes the highest homogeneous part of $h$;\newline
(iii) $\overline{f}\notin k[\overline{g}]$ and $\overline{g}\notin k[%
\overline{f}].$
\end{definition}

\begin{definition}
\textit{(\cite{sh umb1}, Definition 1) }Let $f,g\in \Bbb{C}[x_{1},\ldots
,x_{n}]$ be a *-reduced pair with $\deg f<\deg g.$ Put $p=\frac{\deg f}{\gcd
(\deg f,\deg g)}.$ Then $f,g$ is called a $p$-reduced pair.
\end{definition}

The estimate from Theorem \ref{tw_deg_g_fg} is true even if the condition
(ii) of Definition \ref{def_*-red} is not satisfied. We have the following

\begin{proposition}
\label{prop_deg_g_fg}(\cite{Karas4}, \textit{Proposition 2.6}) Let $f,g\in 
\Bbb{C}[x_{1},\ldots ,x_{n}]$ satisfy conditions (i) and (iii) of Definition 
\ref{def_*-red}. Assume that $\deg f<\deg g,$ put 
\begin{equation*}
p=\frac{\deg f}{\gcd \left( \deg f,\deg g\right) },
\end{equation*}
and let $G(x,y)\in \Bbb{C}[x,y]$ with $\deg _{y}G(x,y)=pq+r,0\leq r<p.$ Then 
\begin{equation*}
\deg G(f,g)\geq q\left( p\deg g-\deg g-\deg f+\deg [f,g]\right) +r\deg g.
\end{equation*}
\end{proposition}

We will also need the following four results from \cite{Karas4}.

\begin{lemma}
\label{lem_degree_linear_change}(\cite{Karas4}, \textit{Lemma 2.7}) If $%
f,g\in \Bbb{C}[x_{1},\ldots ,x_{n}]$ and $L\in GL_{n}(\Bbb{C}),\,$then 
\begin{equation*}
\deg [L^{*}(f),L^{*}(g)]=\deg [f,g],
\end{equation*}
where $L^{*}(h)=h\circ L$ for any $h\in \Bbb{C}[x_{1},\ldots ,x_{n}].$
\end{lemma}

\begin{lemma}
\label{lem_mdeg_linear_change}(\cite{Karas4}, \textit{Lemma 3.16}) For every
mapping $F:\Bbb{C}^{n}\rightarrow \Bbb{C}^{n}$ and every $L\in GL_{n}\left( 
\Bbb{C}\right) $ we have 
\begin{equation*}
\limfunc{mdeg}\left( F\circ L\right) =\limfunc{mdeg}F.
\end{equation*}
\end{lemma}

\begin{theorem}
\label{tw_reduc_type_4}(\cite{Karas4}, \textit{Theorem 3.14}) Let $\left(
d_{1},d_{2},d_{3}\right) \neq \left( 1,1,1\right) ,$ $d_{1}\leq d_{2}\leq
d_{3}$ be a sequence of positive integers. To prove that there is no tame
automorphism $F$ of $\Bbb{C}^{3}$ with $\limfunc{mdeg}F=\left(
d_{1},d_{2},d_{3}\right) $ it is enough to show that a (hypothetical)
automorphism $F$ of $\Bbb{C}^{3}$ with $\limfunc{mdeg}F=\left(
d_{1},d_{2},d_{3}\right) $ admits neither a reduction of type III nor an
elementary reduction. Moreover, if we additionally assume that $\frac{d_{3}}{%
d_{2}}=\frac{3}{2}$ or $3\nmid d_{1},\,$then it is enough to show that no
(hypothetical) automorphism of $\Bbb{C}^{3}$ with multidegree $\left(
d_{1},d_{2},d_{3}\right) $ admits an elementary reduction. In both cases we
can restrict our attention to automorphisms $F:\Bbb{C}^{3}\rightarrow \Bbb{C}%
^{3}$ such that $F\left( 0,0,0\right) =\left( 0,0,0\right) .$
\end{theorem}

Let us recall that an automorphism $F=(F_{1},F_{2},F_{3})$ admits an \textit{%
elementary reduction} if there exists a polynomial $g\in \Bbb{C}[x,y]$ and a
permutation $\sigma $ of $\{1,2,3\}$ such that $\deg (F_{\sigma
(1)}-g(F_{\sigma (2)},F_{\sigma (3)}))<\deg F_{\sigma (1)}.$

Because the above theorem in \cite{Karas4} has number 3.14, we will refer to
it as the $\Pi $-theorem.

\begin{lemma}
\label{lem_deg_poiss_2}(\cite{Karas4}, \textit{Lemma 3.19}) Let $f,g\in \Bbb{%
C}[x_{1},\ldots ,x_{n}]$ be such that 
\begin{equation*}
f=x_{1}+f_{2}+\cdots +f_{n},\qquad g=x_{2}+g_{2}+\cdots +g_{m},
\end{equation*}
where $f_{i},g_{i}$ are homogeneous forms of degree $i.$ If $\deg \left[
f,g\right] =2,$ then $f,g\in \Bbb{C}[x_{1},x_{2}].$
\end{lemma}

The last result that we recall here is the following one due to Moh \cite
{Moh}.

\begin{theorem}
\textit{(see also \cite{van den Essen})} \label{tw_JC_dim2_Moh}Let $%
F:k^{2}\rightarrow k^{2}$ be a Keller map with $\deg F\leq 101.$ Then $F$ is
invertible.
\end{theorem}

\subsection{The proofs}

\begin{proof}
(of Theorem \ref{tw_456_implies}) By the $\Pi $-theorem (Theorem \ref
{tw_reduc_type_4}) it is enough to show that a hypothetical automorphism $F$
of $\Bbb{C}^{3}$ with $\limfunc{mdeg}F=\left( 4,5,6\right) $ does not admit
an elementary reduction. Moreover, it is enough to show this for
automorphisms $F:\Bbb{C}^{3}\rightarrow \Bbb{C}^{3}$ such that $F\left(
0,0,0\right) =\left( 0,0,0\right) .$

So let us assume that there is an automorphism $F=\left(
F_{1},F_{2},F_{3}\right) :\Bbb{C}^{3}\rightarrow \Bbb{C}^{3}$ with $\limfunc{%
mdeg}F=\left( 4,5,6\right) $ such that $F$ admits an elementary reduction of
the form $\left( F_{1},F_{2},F_{3}-g\left( F_{1},F_{2}\right) \right) ,$
where $g\in \Bbb{C}\left[ x,y\right] .$ Then 
\begin{equation}
\deg g\left( F_{1},F_{2}\right) =\deg F_{3}=6,  \label{row_tw_569_1}
\end{equation}
and by Proposition \ref{prop_deg_g_fg}, 
\begin{equation}
\deg g\left( F_{1},F_{2}\right) \geq q\left( 4\cdot 5-5-4+\deg \left[
F_{1},F_{2}\right] \right) +5r,  \label{row_tw_569_2}
\end{equation}
where $\deg _{y}g\left( x,y\right) =4q+r,$ with $0\leq r<4.$ Since $4\cdot
5-5-4+\deg \left[ F_{1},F_{2}\right] \geq 11+\deg \left[ F_{1},F_{2}\right]
>6,$ by (\ref{row_tw_569_1}) and (\ref{row_tw_569_2}) we have $q=0.$ Also by
(\ref{row_tw_569_1}) and (\ref{row_tw_569_2}) we have $r<2.$ Thus $g\left(
x,y\right) =g_{0}\left( x\right) +yg_{0}\left( x\right) .$ And, since $4\Bbb{%
N}\cap \left( 5+4\Bbb{N}\right) =\emptyset ,$ it follows that 
\begin{equation*}
6=\deg g\left( F_{1},F_{2}\right) \in 4\Bbb{N}\cup \left( 5+4\Bbb{N}\right) ,
\end{equation*}
a contradiction.

Now, assume that $F=\left( F_{1},F_{2},F_{3}\right) :\Bbb{C}^{3}\rightarrow 
\Bbb{C}^{3}$ is an automorphism such that $\limfunc{mdeg}F=\left(
4,5,6\right) $ and $F$ admits an elementary reduction of the form $\left(
F_{1}-g\left( F_{2},F_{3}\right) ,\allowbreak F_{2},F_{3}\right) ,$ where $%
g\in \Bbb{C}\left[ x,y\right] .$ Then 
\begin{equation}
\deg g\left( F_{2},F_{3}\right) =\deg F_{1}=4,  \label{row_tw_569_3}
\end{equation}
and by Proposition \ref{prop_deg_g_fg}, 
\begin{equation}
\deg g\left( F_{2},F_{3}\right) \geq q\left( 5\cdot 6-6-5+\deg \left[
F_{1},F_{3}\right] \right) +6r,  \label{row_tw_569_4}
\end{equation}
where $\deg _{y}g\left( x,y\right) =5q+r,$ with $0\leq r<5.$ Since $5\cdot
6-6-5+\deg \left[ F_{1},F_{3}\right] \geq 29+\deg \left[ F_{1},F_{3}\right]
>4,$ we have $q=r=0.$ This means that $g\left( x,y\right) =g\left( x\right)
, $ and so 
\begin{equation*}
4=\deg g\left( F_{2},F_{3}\right) =\deg g\left( F_{2}\right) \in 5\Bbb{N},
\end{equation*}
a contradiction.

Finally, assume that $F=\left( F_{1},F_{2},F_{3}\right) $ is an automorphism
of $\Bbb{C}^{3}$ such that $\limfunc{mdeg}F=\left( 4,5,6\right) $ and $F$
admits an elementary reduction of the form $\left( F_{1},F_{2}-g\left(
F_{1},F_{3}\right) ,F_{3}\right) ,$ where $g\in \Bbb{C}\left[ x,y\right] .$
By Theorem \ref{tw_reduc_type_4} we can also assume that $F\left(
0,0,0\right) =\left( 0,0,0\right) .$ We have 
\begin{equation}
\deg g\left( F_{1},F_{3}\right) =\deg F_{2}=5,  \label{row_tw_569_5}
\end{equation}
and by Proposition \ref{prop_deg_g_fg}, 
\begin{equation}
\deg g\left( F_{1},F_{3}\right) \geq q\left( p\cdot 6-6-4+\deg \left[
F_{2},F_{3}\right] \right) +6r,  \label{row_tw_569_6}
\end{equation}
where $\deg _{y}g\left( x,y\right) =qp+r,$ with $0\leq r<p$ and $p=\frac{4}{%
\gcd \left( 4,6\right) }=2.$ By (\ref{row_tw_569_5}) and (\ref{row_tw_569_6}%
) we see that $r=0.$

Consider the case $\deg \left[ F_{1},F_{3}\right] >3.$ Then $p\cdot
6-6-4+\deg \left[ F_{2},F_{3}\right] =2+\deg \left[ F_{2},F_{3}\right] >5,$
and by (\ref{row_tw_569_5}) and (\ref{row_tw_569_6}) we see that $q=0.$ Thus
in this case, we have $g\left( x,y\right) =g\left( x\right) ,$ and so $\deg
g\left( F_{1},F_{3}\right) =\deg g\left( F_{1}\right) \in 4\Bbb{N}.$ This
contradicts (\ref{row_tw_569_5}). Thus, $\deg \left[ F_{1},F_{3}\right] \leq
3.$

Let $L$ be the linear part of the automorphism $F.$ Since $F\left(
0,0,0\right) =\left( 0,0,0\right) ,$ the linear part of $F\circ L^{-1}$ is
the identity map $\limfunc{id}_{\Bbb{C}^{3}}.$ Thus 
\begin{eqnarray}
F_{1}\circ L^{-1} &=&x+\text{higher degree summands,}  \label{row_tw_569_7}
\\
F_{3}\circ L^{-1} &=&z+\text{higher degree summands.}  \notag
\end{eqnarray}
By Lemma \ref{lem_degree_linear_change}, we have 
\begin{equation*}
\deg \left[ F_{1}\circ L^{-1},F_{3}\circ L^{-1}\right] =\deg \left[
F_{1},F_{3}\right] \leq 3,
\end{equation*}
and by Lemma \ref{lem_mdeg_linear_change} we have $\deg \left( F_{1}\circ
L^{-1}\right) =4,\deg \left( F_{3}\circ L^{-1}\right) =6.$ Thus we can take $%
P=F_{1}\circ L^{-1}$ and $Q=F_{3}\circ L^{-1}.$
\end{proof}

\begin{proof}
(of Theorem \ref{tw_456_implies_rown}) Assume that there is a tame
automorphism of $\Bbb{C}^{3}$ with multidegree $\left( 4,5,6\right) .$ By
Theorem \ref{tw_456_implies} there exists a pair of polynomials $P,Q\in \Bbb{%
C}\left[ x,y,z\right] $ such that 
\begin{eqnarray*}
P &=&x+P_{2}+P_{3}+P_{4},\qquad P_{4}\neq 0, \\
Q &=&z+Q_{2}+\cdots +Q_{6},\qquad Q_{6}\neq 0,
\end{eqnarray*}
and 
\begin{equation*}
\deg \left[ P,Q\right] \leq 3,
\end{equation*}
where $P_{i},Q_{i}$ are homogeneous polynomials of degree $i.$

Since $P$ and $Q$ are algebraically independent (over $\Bbb{C}$), we have $%
\deg \left[ P,Q\right] \geq 2.$ Assume that $\deg \left[ P,Q\right] =2.$
Then, by Lemma \ref{lem_deg_poiss_2}, we have 
\begin{equation*}
P,Q\in \Bbb{C}\left[ x,z\right] .
\end{equation*}
But $\deg \left[ P,Q\right] =2$ means that 
\begin{equation*}
Jac\left( P,Q\right) \in \Bbb{C}^{*}
\end{equation*}
(of course we consider here $P,Q$ as functions of two variables $x,z$). Then
by Theorem \ref{tw_JC_dim2_Moh} the map $\left( P,Q\right) :\Bbb{C}%
^{2}\rightarrow \Bbb{C}^{2}$ is an automorphism. But $4\nmid 6$
contradicting the Jung - van der Kulk theorem (see e.g. \cite{Jung}, \cite
{Kulk} or \cite{van den Essen}). This shows that $\deg \left[ P,Q\right] =3.$
\end{proof}

\section{$H$-reduction method}

In this section we develop the main tool that will be used in the next
sections of the paper. We start with the following lemma in which we use the
notation 
\begin{equation*}
\Bbb{C}[x_{1},\ldots ,x_{n}]_{d}=\left\{ f\in \Bbb{C}[x_{1},\ldots ,x_{n}]:f%
\text{ is homogeneous od degree }d\right\} \cup \{0\}.
\end{equation*}

\begin{lemma}
\label{Lemma_H_reduction}Let $H$ be a squarefree, nonconstant homogeneous
polynomial and let $P$ be any homogeneous polynomial such that 
\begin{equation*}
\left[ H,P\right] =0.
\end{equation*}
Then there exist $a\in \Bbb{C}$ and $k\in \Bbb{N}$ such that 
\begin{equation*}
P=aH^{k}.
\end{equation*}
Moreover, if $P\in \Bbb{C}[x_{1},\ldots ,x_{n}]_{d}$ and $\deg H\nmid d,$
then $P=0.$
\end{lemma}

\begin{proof}
Since $\left[ H,P\right] =0,$ it follows that $H$ and $P$ are algebraically
dependent and so by Lemma 2 in \cite{Umirbaev Yu} there exist $a,b\in \Bbb{C}%
,$ $k_{1},k_{2}\in \Bbb{N}$ and a homogeneous polynomial $h$ such that 
\begin{equation*}
P=ah^{k_{1}}\qquad \text{and\qquad }H=bh^{k_{2}}.
\end{equation*}
Since $H$ is squarefree, we conclude that $k_{2}=1$ and so we can take $h=H.$
Thus, in particular, if $a\neq 0$ (i.e. $P\neq 0$), then $\deg P$ is
divisible by $\deg H.$
\end{proof}

\begin{corollary}
\label{Corr_H_reduction}Let $H$ be a squarefree, nonconstant homogeneous
polynomial and let $P$ be any polynomial such that 
\begin{equation*}
\left[ H,P\right] =0.
\end{equation*}
Then $P\in \Bbb{C}[H].$
\end{corollary}

\begin{proof}
Let $d=\deg P$ and let $P=P_{0}+\cdots +P_{d}$ be the homogeneous
decomposition of $P.$ Since $\left[ H,P\right] =0,$ it follows that 
\begin{equation}
\left[ H,P_{i}\right] =0\qquad \text{for }i=0,\ldots ,d.
\label{Row_H_reduction_1}
\end{equation}
In particular, $\left[ H,P_{d}\right] =0.$ Since $P_{d}\neq 0$ (by
definition of $d$), it follows that $d=k\deg H$ for some $k\in \Bbb{N}.$ By (%
\ref{Row_H_reduction_1}) and Lemma \ref{Lemma_H_reduction} there exist $%
a_{0},\ldots ,a_{k}\in \Bbb{C},$ $a_{k}\neq 0$ such that 
\begin{equation*}
P_{l\deg H}=a_{l}H^{l}\qquad \text{for }l=0,\ldots ,k
\end{equation*}
and $P_{i}=0$ for $i\notin \{0,\deg H,2\deg H,\ldots ,k\deg H\}.$
\end{proof}

We will also use the following fact that is easy to check.

\begin{lemma}
Let $P\in \Bbb{C}[x_{1},\ldots ,x_{n}].$ For any $Q,R\in \Bbb{C}%
[x_{1},\ldots ,x_{n}]$ and $\alpha ,\beta \in \Bbb{C}$ we have 
\begin{eqnarray*}
\left[ P,QR\right] &=&Q\left[ P,R\right] +R\left[ P,Q\right] , \\
\left[ P,\alpha Q+\beta R\right] &=&\alpha \left[ P,Q\right] +\beta \left[
P,R\right] , \\
\left[ P,Q\right] &=&-\left[ Q,P\right] .
\end{eqnarray*}
In other words, the mappings $Q\mapsto \left[ P,Q\right] $ and $Q\mapsto
\left[ Q,P\right] $ are $\Bbb{C}$-derivations.
\end{lemma}

\section{Non-existence of a special pair of polynomials - preliminary lemma}

In this and the next sections our goal is to prove the following theorem.

\begin{theorem}
\label{tw_PQ_deg3}There is no pair of polynomials $F,G\in \Bbb{C}\left[
x,y,z\right] $ such that 
\begin{eqnarray*}
F &=&x+F_{2}+F_{3}+F_{4},\quad F_{4}\neq 0, \\
G &=&z+G_{2}+\cdots +G_{6},\quad G_{6}\neq 0,
\end{eqnarray*}
and 
\begin{equation*}
\deg \left[ F,G\right] \leq 3,
\end{equation*}
where $F_{i},G_{i}$ are homogeneous polynomials of degree $i.$
\end{theorem}

First of all let us notice that the above theorem and Theorem \ref
{tw_456_implies} give Theorem \ref{tw_main_4910}.

Until the end of the paper we assume that 
\begin{equation*}
F=x+F_{2}+F_{3}+F_{4},\quad F_{4}\neq 0,
\end{equation*}
and 
\begin{equation*}
G=z+G_{2}+\cdots +G_{6},\quad G_{6}\neq 0.
\end{equation*}

The main idea in proving Theorem \ref{tw_PQ_deg3} is to use $H$-reduction to
show that smaller $\deg \left[ F,G\right] $ gives a closer relation between $%
F$ and $G.$ In other words, smaller $\deg \left[ F,G\right] $ implies
smaller flexibility in choosing $F$ and $G.$ And finally, small enough $\deg
\left[ F,G\right] $ implies that there is no space for $F$ and $G.$ The
first step is the following lemma.

\begin{lemma}
\label{Lem_deg_10}If $\deg \left[ F,G\right] <10,$ then either\newline
(1) there is a squarefree homogeneous polynomial $H$ of degree $2$ and $%
\alpha \in \Bbb{C}^{*}$ such that 
\begin{equation*}
F_{4}=H^{2},\qquad G_{6}=\alpha H^{3},
\end{equation*}
\newline
or\newline
(2) there is a homogeneous polynomial $h$ of degree $1$ and $\alpha \in \Bbb{%
C}^{*}$ such that 
\begin{equation*}
F_{4}=h^{4},\qquad G_{6}=\alpha h^{6}.
\end{equation*}
\end{lemma}

\begin{proof}
Since $\deg \left[ F,G\right] <10,$ we have $\left[ F_{4},G_{6}\right] =0,$
and so $F_{4}$ and $G_{6}$ are algebraically dependent. Thus there is a
homogeneous polynomial $\widetilde{H},$ $a,\alpha \in \Bbb{C}^{*}$ and $%
k_{1},k_{2}\in \Bbb{N}^{*}$ such that 
\begin{equation*}
F_{4}=a\widetilde{H}^{k_{1}},\qquad G_{6}=\alpha \widetilde{H}^{k_{2}}.
\end{equation*}
Because $\Bbb{C}$ is algebraically closed we can assume that $a=1.$ Since $%
\gcd \left( 4,6\right) =2,$ there are two possibilities: $k_{1}=2,$ $k_{2}=3$
or $k_{1}=4,$ $k_{2}=6.$ In the second case we take $h=\widetilde{H}.$ And
in the first case, $\widetilde{H}$ is either squarefree or not. If it is
squarefree we take $H=\widetilde{H}.$ And if it is not squarefree then there
exist a homogeneous polynomial $h$ of degree $1$ and $\gamma \in \Bbb{C}^{*}$
such that $\widetilde{H}=\gamma h^{2}.$ But since $\Bbb{C}$ is algebraically
closed we can assume that $\gamma =1.$
\end{proof}

\section{The case of squarefree $H$}

Now we consider the situation of Lemma \ref{Lem_deg_10}(1).

\begin{lemma}
\label{Lem_deg_9_H}Let $\deg \left[ F,G\right] <9$ and let $\alpha $ and $H$
be as in Lemma \ref{Lem_deg_10}(1). Then 
\begin{equation*}
G_{5}=\frac{3}{2}\alpha HF_{3}.
\end{equation*}
\end{lemma}

\begin{proof}
Since $\deg \left[ F,G\right] <9,$ it follows that 
\begin{equation*}
\left[ F_{4},G_{5}\right] +\left[ F_{3},G_{6}\right] =0.
\end{equation*}
By Lemma \ref{Lem_deg_10}(1), 
\begin{eqnarray*}
\left[ F_{4},G_{5}\right] +\left[ F_{3},G_{6}\right] &=&\left[
H^{2},G_{5}\right] +\left[ F_{3},\alpha H^{3}\right] \\
&=&2H\left[ H,G_{5}\right] +3\alpha H^{2}\left[ F_{3},H\right] \\
&=&2H\left[ H,G_{5}\right] -3\alpha H^{2}\left[ H,F_{3}\right] .
\end{eqnarray*}
Since $H$ is a constant for the derivation $P\mapsto \left[ H,P\right] ,$ we
see that 
\begin{equation*}
2H\left[ H,G_{5}\right] -3\alpha H^{2}\left[ H,F_{3}\right] =\left[
H,2HG_{5}-3\alpha H^{2}F_{3}\right] .
\end{equation*}
Thus $\left[ H,2HG_{5}-3\alpha H^{2}F_{3}\right] =0,$ and since $%
2HG_{5}-3\alpha H^{2}F_{3}\in \Bbb{C}[x,y,z]_{7},$ we have, by Lemma \ref
{Lemma_H_reduction}, $2HG_{5}-3\alpha H^{2}F_{3}=0.$
\end{proof}

\begin{lemma}
\label{Lem_deg_8_H}Let $\deg \left[ F,G\right] <8$ and let $\alpha $ and $H$
be as in Lemma \ref{Lem_deg_9_H}. Then there is a homogeneous polynomial $%
\widetilde{F}_{1}$ of degree $1$ and $b\in \Bbb{C}$ such that 
\begin{eqnarray*}
F_{3} &=&H\widetilde{F}_{1},\qquad G_{5}=\frac{3}{2}\alpha H^{2}\widetilde{F}%
_{1}, \\
G_{4} &=&\frac{3}{8}\alpha H\widetilde{F}_{1}^{2}+\frac{3}{2}\alpha
HF_{2}+bH^{2}.
\end{eqnarray*}
\end{lemma}

\begin{proof}
Since $\deg \left[ F,G\right] <8,$ we have 
\begin{equation*}
\left[ F_{4},G_{4}\right] +\left[ F_{3},G_{5}\right] +\left[
F_{2},G_{6}\right] =0.
\end{equation*}
By Lemma \ref{Lem_deg_9_H}, 
\begin{equation*}
\left[ F_{4},G_{4}\right] =\left[ H^{2},G_{4}\right] =2H\left[
H,G_{4}\right] =\left[ H,2HG_{4}\right] ,
\end{equation*}
\begin{equation*}
\left[ F_{3},G_{5}\right] =\left[ F_{3},\frac{3}{2}\alpha HF_{3}\right] =%
\frac{3}{2}\alpha F_{3}\left[ F_{3},H\right] =-\left[ H,\frac{3}{4}\alpha
F_{3}^{2}\right]
\end{equation*}
and 
\begin{equation*}
\left[ F_{2},G_{6}\right] =\left[ F_{2},\alpha H^{3}\right] =3\alpha
H^{2}\left[ F_{2},H\right] =-\left[ H,3\alpha H^{2}F_{2}\right] .
\end{equation*}
Thus $\left[ H,2HG_{4}-\frac{3}{4}\alpha F_{3}^{2}-3\alpha H^{2}F_{2}\right]
=0,$ and so there exists $b\in \Bbb{C}$ such that (see Lemma \ref
{Lemma_H_reduction}) 
\begin{equation}
2HG_{4}-\frac{3}{4}\alpha F_{3}^{2}-3\alpha H^{2}F_{2}=2bH^{3}.
\label{Row_lem_deg_8_H_1}
\end{equation}
Since $H|2HG_{4}-3\alpha H^{2}F_{2}$ and $\alpha \neq 0,$ we conclude that $%
H|F_{3}^{2}.$ Since $H$ is squarefree, it follows that $H|F_{3}.$ Thus there
exists a homogeneous polynomial $\widetilde{F}_{1}$ such that $F_{3}=H%
\widetilde{F}_{1}.$ Now (\ref{Row_lem_deg_8_H_1}) can be written as follows: 
\begin{equation*}
2HG_{4}-\frac{3}{4}\alpha H^{2}\widetilde{F}_{1}^{2}-3\alpha
H^{2}F_{2}=2bH^{3}.
\end{equation*}
\end{proof}

\begin{lemma}
\label{Lem_deg_7_H}Let $\deg \left[ F,G\right] <7$ and let $\alpha ,b,H,%
\widetilde{F}_{1}$ be as in Lemma \ref{Lem_deg_8_H}. Then 
\begin{equation*}
G_{3}=-\frac{1}{16}\alpha \widetilde{F}_{1}^{3}+bH\widetilde{F}_{1}+\frac{3}{%
2}\alpha Hx+\frac{3}{4}\alpha \widetilde{F}_{1}F_{2}.
\end{equation*}
\end{lemma}

\begin{proof}
Since $\deg \left[ F,G\right] <7,$ we see that 
\begin{equation}
\left[ F_{4},G_{3}\right] +\left[ F_{3},G_{4}\right] +\left[
F_{2},G_{5}\right] +\left[ x,G_{6}\right] =0.  \label{Row_lem_deg_7_H_1}
\end{equation}
By Lemma \ref{Lem_deg_9_H}, 
\begin{equation}
\left[ F_{4},G_{3}\right] =\left[ H^{2},G_{3}\right] =2H\left[
H,G_{3}\right] =\left[ H,2HG_{3}\right] ,  \label{Row_lem_deg_7_H_2}
\end{equation}
and by Lemma \ref{Lem_deg_8_H}, 
\begin{eqnarray}
\left[ F_{3},G_{4}\right] &=&\left[ H\widetilde{F}_{1},\frac{3}{8}\alpha H%
\widetilde{F}_{1}^{2}+\frac{3}{2}\alpha HF_{2}+bH^{2}\right]
\label{Row_lem_deg_7_H_3} \\
&=&H\left[ \widetilde{F}_{1},\frac{3}{8}\alpha H\widetilde{F}_{1}^{2}+\frac{3%
}{2}\alpha HF_{2}+bH^{2}\right]  \notag \\
&&+\widetilde{F}_{1}\left[ H,\frac{3}{8}\alpha H\widetilde{F}_{1}^{2}+\frac{3%
}{2}\alpha HF_{2}+bH^{2}\right]  \notag \\
&=&\underset{----------}{\frac{3}{8}\alpha H\widetilde{F}_{1}^{2}\left[ 
\widetilde{F}_{1},H\right] }+\underset{==========}{\frac{3}{2}\alpha
H^{2}\left[ \widetilde{F}_{1},F_{2}\right] }+\underset{---+---+---}{\frac{3}{%
2}\alpha HF_{2}\left[ \widetilde{F}_{1},H\right] }  \notag \\
&&+2bH^{2}\left[ \widetilde{F}_{1},H\right] +\underset{---------}{\frac{3}{4}%
\alpha H\widetilde{F}_{1}^{2}\left[ H,\widetilde{F}_{1}\right] }+\underset{%
---+---+---}{\frac{3}{2}\alpha H\widetilde{F}_{1}\left[ H,F_{2}\right] }, 
\notag
\end{eqnarray}
\begin{eqnarray}
\left[ F_{2},G_{5}\right] &=&\left[ F_{2},\frac{3}{2}\alpha H^{2}\widetilde{F%
}_{1}\right]  \label{Row_lem_deg_7_H_4} \\
&=&\underset{===========}{\frac{3}{2}\alpha H^{2}\left[ F_{2},\widetilde{F}%
_{1}\right] }+\underset{---+---+---}{3\alpha H\widetilde{F}_{1}\left[
F_{2},H\right] },  \notag
\end{eqnarray}
\begin{equation}
\left[ x,G_{6}\right] =\left[ x,\alpha H^{3}\right] =3\alpha H^{2}\left[
x,H\right] =\left[ H,-3\alpha H^{2}x\right] .  \label{Row_lem_deg_7_H_5}
\end{equation}
Notice that: 
\begin{equation}
\underset{=======================}{\frac{3}{2}\alpha H^{2}\left[ \widetilde{F%
}_{1},F_{2}\right] +\frac{3}{2}\alpha H^{2}\left[ F_{2},\widetilde{F}%
_{1}\right] }=0,  \label{Row_lem_deg_7_H_6}
\end{equation}
\begin{equation}
\underset{-----------------------}{\frac{3}{8}\alpha H\widetilde{F}%
_{1}^{2}\left[ \widetilde{F}_{1},H\right] +\frac{3}{4}\alpha H\widetilde{F}%
_{1}^{2}\left[ H,\widetilde{F}_{1}\right] }=\frac{3}{8}\alpha H\widetilde{F}%
_{1}^{2}\left[ H,\widetilde{F}_{1}\right] =\left[ H,\frac{1}{8}\alpha H%
\widetilde{F}_{1}^{3}\right] ,  \label{Row_lem_deg7_H_6b}
\end{equation}
\begin{eqnarray}
&&\underset{---+---+---+---+---+---+---+---+---}{\frac{3}{2}\alpha
HF_{2}\left[ \widetilde{F}_{1},H\right] +\frac{3}{2}\alpha H\widetilde{F}%
_{1}\left[ H,F_{2}\right] +3\alpha H\widetilde{F}_{1}\left[ F_{2},H\right] }
\label{Row_lem_deg_7_H_7} \\
&=&\frac{3}{2}\alpha H\left( F_{2}\left[ \widetilde{F}_{1},H\right] +%
\widetilde{F}_{1}\left[ F_{2},H\right] \right) =\left[ \frac{3}{2}\alpha H%
\widetilde{F}_{1}F_{2},H\right] ,  \notag
\end{eqnarray}
By (\ref{Row_lem_deg_7_H_1})-(\ref{Row_lem_deg_7_H_7}) 
\begin{equation*}
\left[ H,2HG_{3}+\frac{1}{8}\alpha H\widetilde{F}_{1}^{3}-2bH^{2}\widetilde{F%
}_{1}-3\alpha H^{2}x-\frac{3}{2}\alpha H\widetilde{F}_{1}F_{2}\right] =0.
\end{equation*}
Since $2HG_{3}+\frac{1}{8}\alpha H\widetilde{F}_{1}^{3}-2bH^{2}\widetilde{F}%
_{1}-3\alpha H^{2}x-\frac{3}{2}\alpha H\widetilde{F}_{1}F_{2}\in \Bbb{C}%
[x,y,z]_{5},$ we conclude that (see Lemma \ref{Lemma_H_reduction}) 
\begin{equation*}
2HG_{3}+\frac{1}{8}\alpha H\widetilde{F}_{1}^{3}-2bH^{2}\widetilde{F}%
_{1}-3\alpha H^{2}x-\frac{3}{2}\alpha H\widetilde{F}_{1}F_{2}=0.
\end{equation*}
This gives the formula for $G_{3.}$
\end{proof}

\begin{lemma}
\label{Lem_deg_6_H}Let $\deg \left[ F,G\right] <6$ and let $\alpha ,b,H,%
\widetilde{F}_{1}$ be as in Lemma \ref{Lem_deg_7_H}. Then there exist $%
c,d\in \Bbb{C}$ such that 
\begin{eqnarray*}
F_{2} &=&\frac{1}{4}\left( \widetilde{F}_{1}^{2}+dH\right) , \\
G_{4} &=&\frac{3}{4}\alpha H\widetilde{F}_{1}^{2}+\left( \frac{3}{8}\alpha
d+b\right) H^{2}, \\
G_{3} &=&\frac{1}{8}\alpha \widetilde{F}_{1}^{3}+\left( b+\frac{3}{16}\alpha
d\right) H\widetilde{F}_{1}+\frac{3}{2}\alpha Hx, \\
G_{2} &=&AH-\frac{1}{4}b\widetilde{F}_{1}^{2}+\frac{3}{4}\alpha x\widetilde{F%
}_{1},
\end{eqnarray*}
where $A=\frac{3}{128}\alpha d^{2}-\frac{1}{4}bd+\frac{1}{2}c.$
\end{lemma}

\begin{proof}
Since $\deg \left[ F,G\right] <6,$ we have 
\begin{equation}
\left[ F_{4},G_{2}\right] +\left[ F_{3},G_{3}\right] +\left[
F_{2},G_{4}\right] +\left[ x,G_{5}\right] =0.  \label{Row_lem_deg_6_H_1}
\end{equation}

By Lemma \ref{Lem_deg_9_H}, 
\begin{equation}
\left[ F_{4},G_{2}\right] =\left[ H^{2},G_{2}\right] =2H\left[
H,G_{2}\right] =\left[ H,2HG_{2}\right] ,  \label{Row_lem_deg_6_H_2}
\end{equation}
and by Lemmas \ref{Lem_deg_8_H} and \ref{Lem_deg_7_H}, 
\begin{eqnarray}
&&\left[ F_{3},G_{3}\right]  \label{Row_lem_deg_6_H_3} \\
&=&\left[ H\widetilde{F}_{1},-\frac{1}{16}\alpha \widetilde{F}_{1}^{3}+bH%
\widetilde{F}_{1}+\frac{3}{2}\alpha Hx+\frac{3}{4}\alpha \widetilde{F}%
_{1}F_{2}\right]  \notag \\
&=&\left[ H\widetilde{F}_{1},-\frac{1}{16}\alpha \widetilde{F}_{1}^{3}+\frac{%
3}{2}\alpha Hx+\frac{3}{4}\alpha \widetilde{F}_{1}F_{2}\right]  \notag \\
&=&H\left[ \widetilde{F}_{1},-\frac{1}{16}\alpha \widetilde{F}_{1}^{3}+\frac{%
3}{2}\alpha Hx+\frac{3}{4}\alpha \widetilde{F}_{1}F_{2}\right]  \notag \\
&&+\widetilde{F}_{1}\left[ H,-\frac{1}{16}\alpha \widetilde{F}_{1}^{3}+\frac{%
3}{2}\alpha Hx+\frac{3}{4}\alpha \widetilde{F}_{1}F_{2}\right]  \notag \\
&=&\underset{---------}{\frac{3}{2}\alpha Hx\left[ \widetilde{F}%
_{1},H\right] }+\underset{--+--+--+--}{\frac{3}{2}\alpha H^{2}\left[ 
\widetilde{F}_{1},x\right] }+\underset{===========}{\frac{3}{4}\alpha H%
\widetilde{F}_{1}\left[ \widetilde{F}_{1},F_{2}\right] }  \notag \\
&&-\frac{3}{16}\alpha \widetilde{F}_{1}^{3}\left[ H,\widetilde{F}_{1}\right]
+\underset{---------}{\frac{3}{2}\alpha H\widetilde{F}_{1}\left[ H,x\right] }%
+\underset{\symbol{94}\symbol{94}\symbol{94}\symbol{94}\symbol{94}\symbol{94}%
\symbol{94}\symbol{94}\symbol{94}\symbol{94}\symbol{94}\symbol{94}}{\frac{3}{%
4}\alpha \widetilde{F}_{1}^{2}\left[ H,F_{2}\right] }+\underset{\symbol{94}%
\symbol{94}\symbol{94}\symbol{94}\symbol{94}\symbol{94}\symbol{94}\symbol{94}%
\symbol{94}\symbol{94}\symbol{94}\symbol{94}}{\frac{3}{4}\alpha \widetilde{F}%
_{1}F_{2}\left[ H,\widetilde{F}_{1}\right] },  \notag
\end{eqnarray}
\begin{eqnarray}
\left[ F_{2},G_{4}\right] &=&\left[ F_{2},\frac{3}{8}\alpha H\widetilde{F}%
_{1}^{2}+\frac{3}{2}HF_{2}+bH^{2}\right]  \label{Row_lem_deg_6_H_4} \\
&=&\underset{=============}{\frac{3}{4}\alpha H\widetilde{F}_{1}\left[ F_{2},%
\widetilde{F}_{1}\right] }+\underset{\symbol{94}\symbol{94}\symbol{94}%
\symbol{94}\symbol{94}\symbol{94}\symbol{94}\symbol{94}\symbol{94}\symbol{94}%
\symbol{94}\symbol{94}}{\frac{3}{8}\alpha \widetilde{F}_{1}^{2}\left[
F_{2},H\right] }  \notag \\
&&+\frac{3}{2}\alpha F_{2}\left[ F_{2},H\right] +2bH\left[ F_{2},H\right] , 
\notag
\end{eqnarray}
\begin{equation}
\left[ x,G_{5}\right] =\left[ x,\frac{3}{2}\alpha H^{2}\widetilde{F}%
_{1}\right] =\underset{--+--+--}{\frac{3}{2}\alpha H^{2}\left[ x,\widetilde{F%
}_{1}\right] }+\underset{--------}{3\alpha H\widetilde{F}_{1}\left[
x,H\right] }.  \label{Row_lem_deg_6_H_5}
\end{equation}
Notice that: 
\begin{equation}
\underset{======================}{\frac{3}{4}\alpha H\widetilde{F}_{1}\left[ 
\widetilde{F}_{1},F_{2}\right] +\frac{3}{4}\alpha H\widetilde{F}_{1}\left[
F_{2},\widetilde{F}_{1}\right] }=0,  \label{Row_lem_deg_6_H_6}
\end{equation}
\begin{equation}
\underset{--+--+--+--+--+--+--}{\frac{3}{2}\alpha H^{2}\left[ \widetilde{F}%
_{1},x\right] +\frac{3}{2}\alpha H^{2}\left[ x,\widetilde{F}_{1}\right] }=0,
\label{Row_lem_deg_6_H_7}
\end{equation}
\begin{eqnarray}
&&\underset{\symbol{94}\symbol{94}\symbol{94}\symbol{94}\symbol{94}\symbol{94%
}\symbol{94}\symbol{94}\symbol{94}\symbol{94}\symbol{94}\symbol{94}\symbol{94%
}\symbol{94}\symbol{94}\symbol{94}\symbol{94}\symbol{94}\symbol{94}\symbol{94%
}\symbol{94}\symbol{94}\symbol{94}\symbol{94}\symbol{94}\symbol{94}\symbol{94%
}\symbol{94}\symbol{94}\symbol{94}\symbol{94}\symbol{94}\symbol{94}\symbol{94%
}\symbol{94}\symbol{94}\symbol{94}\symbol{94}}{\frac{3}{4}\alpha \widetilde{F%
}_{1}^{2}\left[ H,F_{2}\right] +\frac{3}{4}\alpha \widetilde{F}%
_{1}F_{2}\left[ H,\widetilde{F}_{1}\right] +\frac{3}{8}\alpha \widetilde{F}%
_{1}^{2}\left[ F_{2},H\right] }  \label{Row_lem_deg_6_H_8} \\
&=&\frac{3}{8}\alpha \left( \widetilde{F}_{1}^{2}\left[ H,F_{2}\right] +2%
\widetilde{F}_{1}F_{2}\left[ H,\widetilde{F}_{1}\right] \right) =\left[ H,%
\frac{3}{8}\alpha \widetilde{F}_{1}^{2}F_{2}\right] ,  \notag
\end{eqnarray}
\begin{eqnarray}
&&\underset{--------------------------------}{\frac{3}{2}\alpha Hx\left[ 
\widetilde{F}_{1},H\right] +\frac{3}{2}\alpha H\widetilde{F}_{1}\left[
H,x\right] +3\alpha H\widetilde{F}_{1}\left[ x,H\right] }
\label{Row_lem_deg_6_H_9} \\
&=&\frac{3}{2}\alpha H\left( x\left[ \widetilde{F}_{1},H\right] +\widetilde{F%
}_{1}\left[ x,H\right] \right) =\left[ \frac{3}{2}\alpha Hx\widetilde{F}%
_{1},H\right] .  \notag
\end{eqnarray}
By (\ref{Row_lem_deg_6_H_1})-(\ref{Row_lem_deg_6_H_9}) we have 
\begin{equation*}
\left[ H,2HG_{2}-\frac{3}{64}\alpha \widetilde{F}_{1}^{4}-\frac{3}{4}\alpha
F_{2}^{2}+2bHF_{2}+\frac{3}{8}\alpha \widetilde{F}_{1}^{2}F_{2}-\frac{3}{2}%
\alpha Hx\widetilde{F}_{1}\right] =0.
\end{equation*}
Thus there exists $c\in \Bbb{C}$ such that 
\begin{equation}
2HG_{2}-\frac{3}{64}\alpha \widetilde{F}_{1}^{4}-\frac{3}{4}\alpha
F_{2}^{2}+2bHF_{2}+\frac{3}{8}\alpha \widetilde{F}_{1}^{2}F_{2}-\frac{3}{2}%
\alpha Hx\widetilde{F}_{1}=cH^{2}.  \label{Row_lem_deg_6_H_10}
\end{equation}
Since $H|2HG_{2}+2bHF_{2}-\frac{3}{2}\alpha Hx\widetilde{F}_{1},$ we
conclude that 
\begin{equation*}
H|-\frac{3}{64}\alpha \left( \widetilde{F}_{1}^{4}-8\widetilde{F}%
_{1}^{2}F_{2}+16F_{2}^{2}\right) =-\frac{3}{64}\alpha \left( \widetilde{F}%
_{1}^{2}-4F_{2}\right) ^{2}.
\end{equation*}
Then $H|\widetilde{F}_{1}^{2}-4F_{2},$ because $H$ is squarefree and $\alpha
\neq 0.$ Thus there is $d\in \Bbb{C}$ such that $\widetilde{F}%
_{1}^{2}-4F_{2}=-dH$ or equivalently 
\begin{equation}
F_{2}=\frac{1}{4}\left( \widetilde{F}_{1}^{2}+dH\right) .
\label{Row_lem_deg_6_H_11}
\end{equation}
Using (\ref{Row_lem_deg_6_H_11}) we can rewrite (\ref{Row_lem_deg_6_H_10})
as 
\begin{equation*}
2HG_{2}-\frac{3}{64}\alpha d^{2}H^{2}+2bHF_{2}-\frac{3}{2}\alpha Hx%
\widetilde{F}_{1}=cH^{2}.
\end{equation*}
The last equality and (\ref{Row_lem_deg_6_H_11}) give the formula for $%
G_{2}. $

Lemma \ref{Lem_deg_7_H} and (\ref{Row_lem_deg_6_H_11}) also give 
\begin{eqnarray*}
G_{3} &=&-\frac{1}{16}\alpha \widetilde{F}_{1}^{3}+bH\widetilde{F}_{1}+\frac{%
3}{2}\alpha Hx+\frac{3}{16}\alpha \widetilde{F}_{1}^{3}+\frac{3}{16}\alpha d%
\widetilde{F}_{1}H \\
&=&\frac{1}{8}\alpha \widetilde{F}_{1}^{3}+\left( b+\frac{3}{16}\alpha
d\right) H\widetilde{F}_{1}+\frac{3}{2}\alpha Hx
\end{eqnarray*}
and Lemma \ref{Lem_deg_8_H} and (\ref{Row_lem_deg_6_H_11}) give 
\begin{eqnarray*}
G_{4} &=&\frac{3}{8}\alpha H\widetilde{F}_{1}^{2}+\frac{3}{8}\alpha H%
\widetilde{F}_{1}^{2}+\frac{3}{8}\alpha dH^{2}+bH^{2} \\
&=&\frac{3}{4}\alpha H\widetilde{F}_{1}^{2}+\left( \frac{3}{8}\alpha
d+b\right) H^{2}.
\end{eqnarray*}
\end{proof}

\begin{lemma}
\label{Lem_deg_5_H}Let $\deg \left[ F,G\right] <5$ and let $\alpha ,b,c,d,H,%
\widetilde{F}_{1}$ be as in Lemma \ref{Lem_deg_6_H}. Then 
\begin{eqnarray*}
b &=&0, \\
G_{4} &=&\frac{3}{4}\alpha H\widetilde{F}_{1}^{2}+\frac{3}{8}\alpha dH^{2},
\\
G_{3} &=&\frac{1}{8}\alpha \widetilde{F}_{1}^{3}+\frac{3}{16}\alpha dH%
\widetilde{F}_{1}+\frac{3}{2}\alpha Hx, \\
G_{2} &=&AH+\frac{3}{4}\alpha x\widetilde{F}_{1}, \\
z &=&M\widetilde{F}_{1}+\frac{3}{16}\alpha dx,
\end{eqnarray*}
where $M=-\frac{3}{256}\alpha d^{2}+\frac{1}{4}c.$
\end{lemma}

\begin{proof}
Since $\deg \left[ F,G\right] <5,$ we see that 
\begin{equation}
\left[ F_{4},z\right] +\left[ F_{3},G_{2}\right] +\left[ F_{2},G_{3}\right]
+\left[ x,G_{4}\right] =0.  \label{Row_lem_deg_5_H_1}
\end{equation}
By Lemma \ref{Lem_deg_9_H}, 
\begin{equation}
\left[ F_{4},z\right] =\left[ H^{2},z\right] =2H\left[ H,z\right] =\left[
H,2Hz\right] ,  \label{Row_lem_deg_5_H_1b}
\end{equation}
and by Lemmas \ref{Lem_deg_8_H} and \ref{Lem_deg_6_H}, 
\begin{eqnarray}
\left[ F_{3},G_{2}\right]  &=&\left[ H\widetilde{F}_{1},AH-\frac{1}{4}b%
\widetilde{F}_{1}^{2}+\frac{3}{4}\alpha x\widetilde{F}_{1}\right] 
\label{Row_lem_deg_5_H_2} \\
&=&H\left[ \widetilde{F}_{1},AH-\frac{1}{4}b\widetilde{F}_{1}^{2}+\frac{3}{4}%
\alpha x\widetilde{F}_{1}\right]   \notag \\
&&+\widetilde{F}_{1}\left[ H,AH-\frac{1}{4}b\widetilde{F}_{1}^{2}+\frac{3}{4}%
\alpha x\widetilde{F}_{1}\right]   \notag \\
&=&AH\left[ \widetilde{F}_{1},H\right] +\underset{----------}{\frac{3}{4}%
\alpha H\widetilde{F}_{1}\left[ \widetilde{F}_{1},x\right] }  \notag \\
&&\underset{\symbol{94}\symbol{94}\symbol{94}\symbol{94}\symbol{94}\symbol{94%
}\symbol{94}\symbol{94}\symbol{94}\symbol{94}\symbol{94}}{-\frac{1}{2}b%
\widetilde{F}_{1}^{2}\left[ H,\widetilde{F}_{1}\right] }+\underset{%
===========}{\frac{3}{4}\alpha x\widetilde{F}_{1}\left[ H,\widetilde{F}%
_{1}\right] }+\underset{--+--+--}{\frac{3}{4}\alpha \widetilde{F}%
_{1}^{2}\left[ H,x\right] },  \notag
\end{eqnarray}
\begin{eqnarray}
&&\left[ F_{2},G_{3}\right]   \label{Row_lem_deg_5_H_3} \\
&=&\left[ \frac{1}{4}\widetilde{F}_{1}^{2}+\frac{1}{4}dH,\frac{1}{8}\alpha 
\widetilde{F}_{1}^{3}+\left( b+\frac{3}{16}\alpha d\right) H\widetilde{F}%
_{1}+\frac{3}{2}\alpha Hx\right]   \notag \\
&=&\frac{1}{2}\widetilde{F}_{1}\left[ \widetilde{F}_{1},\frac{1}{8}\alpha 
\widetilde{F}_{1}^{3}+\left( b+\frac{3}{16}\alpha d\right) H\widetilde{F}%
_{1}+\frac{3}{2}\alpha Hx\right]   \notag \\
&&+\frac{1}{4}d\left[ H,\frac{1}{8}\alpha \widetilde{F}_{1}^{3}+\left( b+%
\frac{3}{16}\alpha d\right) H\widetilde{F}_{1}+\frac{3}{2}\alpha Hx\right]  
\notag \\
&=&\underset{\symbol{94}\symbol{94}\symbol{94}\symbol{94}\symbol{94}\symbol{%
94}\symbol{94}\symbol{94}\symbol{94}\symbol{94}\symbol{94}\symbol{94}\symbol{%
94}\symbol{94}\symbol{94}\symbol{94}\symbol{94}\symbol{94}\symbol{94}}{%
\left( \frac{1}{2}b+\frac{3}{32}\alpha d\right) \widetilde{F}_{1}^{2}\left[ 
\widetilde{F}_{1},H\right] }+\underset{----------}{\frac{3}{4}\alpha H%
\widetilde{F}_{1}\left[ \widetilde{F}_{1},x\right] }+\underset{===========}{%
\frac{3}{4}\alpha x\widetilde{F}_{1}\left[ \widetilde{F}_{1},H\right] } 
\notag \\
&&+\underset{\symbol{94}\symbol{94}\symbol{94}\symbol{94}\symbol{94}\symbol{%
94}\symbol{94}\symbol{94}\symbol{94}\symbol{94}\symbol{94}}{\frac{3}{32}%
\alpha d\widetilde{F}_{1}^{2}\left[ H,\widetilde{F}_{1}\right] }+\left( 
\frac{1}{4}bd+\frac{3}{64}\alpha d^{2}\right) H\left[ H,\widetilde{F}%
_{1}\right] +\underset{++++++++}{\frac{3}{8}\alpha dH\left[ H,x\right] }, 
\notag
\end{eqnarray}
\begin{eqnarray}
\left[ x,G_{4}\right]  &=&\left[ x,\frac{3}{4}\alpha H\widetilde{F}%
_{1}^{2}+\left( \frac{3}{8}\alpha d+b\right) H^{2}\right] 
\label{Row_lem_deg_5_H_4} \\
&=&\underset{--+--+--}{\frac{3}{4}\alpha \widetilde{F}_{1}^{2}\left[
x,H\right] }+\underset{----------}{\frac{3}{2}\alpha H\widetilde{F}%
_{1}\left[ x,\widetilde{F}_{1}\right] }+\underset{++++++++++++++++}{\left( 
\frac{3}{4}\alpha d+2b\right) H\left[ x,H\right] }.  \notag
\end{eqnarray}
Notice that: 
\begin{equation}
\underset{------------------------------}{\frac{3}{4}\alpha H\widetilde{F}%
_{1}\left[ \widetilde{F}_{1},x\right] +\frac{3}{4}\alpha H\widetilde{F}%
_{1}\left[ \widetilde{F}_{1},x\right] +\frac{3}{2}\alpha H\widetilde{F}%
_{1}\left[ x,\widetilde{F}_{1}\right] }=0,  \label{Row_lem_deg_5_H_5}
\end{equation}
\begin{equation}
\underset{======================}{\frac{3}{4}\alpha x\widetilde{F}_{1}\left[
H,\widetilde{F}_{1}\right] +\frac{3}{4}\alpha x\widetilde{F}_{1}\left[ 
\widetilde{F}_{1},H\right] }=0,  \label{Row_lem_deg_5_H_6}
\end{equation}
\begin{equation}
\underset{--+--+--+--+--+--+--}{\frac{3}{4}\alpha \widetilde{F}%
_{1}^{2}\left[ H,x\right] +\frac{3}{4}\alpha \widetilde{F}_{1}^{2}\left[
x,H\right] }=0,  \label{Row_lem_deg_5_H_7}
\end{equation}
\begin{eqnarray}
&&\underset{++++++++++++++++++++++}{\frac{3}{8}\alpha dH\left[ H,x\right]
+\left( \frac{3}{4}\alpha d+2b\right) H\left[ x,H\right] }
\label{Row_lem_deg_5_H_8} \\
&=&-\left[ H,\left( \frac{3}{8}\alpha d+2b\right) Hx\right] ,  \notag
\end{eqnarray}
\begin{eqnarray}
&&\underset{\symbol{94}\symbol{94}\symbol{94}\symbol{94}\symbol{94}\symbol{94%
}\symbol{94}\symbol{94}\symbol{94}\symbol{94}\symbol{94}\symbol{94}\symbol{94%
}\symbol{94}\symbol{94}\symbol{94}\symbol{94}\symbol{94}\symbol{94}\symbol{94%
}\symbol{94}\symbol{94}\symbol{94}\symbol{94}\symbol{94}\symbol{94}\symbol{94%
}\symbol{94}\symbol{94}\symbol{94}\symbol{94}\symbol{94}\symbol{94}\symbol{94%
}\symbol{94}\symbol{94}\symbol{94}\symbol{94}\symbol{94}\symbol{94}\symbol{94%
}\symbol{94}\symbol{94}\symbol{94}}{-\frac{1}{2}b\widetilde{F}_{1}^{2}\left[
H,\widetilde{F}_{1}\right] +\left( \frac{1}{2}b+\frac{3}{32}\alpha d\right) 
\widetilde{F}_{1}^{2}\left[ \widetilde{F}_{1},H\right] +\frac{3}{32}\alpha d%
\widetilde{F}_{1}^{2}\left[ H,\widetilde{F}_{1}\right] }
\label{Row_lem_deg_5_H_9} \\
&=&\left[ H,-\frac{1}{6}b\widetilde{F}_{1}^{3}\right] -\left[ H,\left( \frac{%
1}{6}b+\frac{1}{32}\alpha d\right) \widetilde{F}_{1}^{3}\right] +\left[ H,%
\frac{1}{32}\alpha d\widetilde{F}_{1}^{3}\right]   \notag \\
&=&\left[ H,-\frac{1}{3}b\widetilde{F}_{1}^{3}\right]   \notag
\end{eqnarray}
By (\ref{Row_lem_deg_5_H_1})-(\ref{Row_lem_deg_5_H_9}) we have 
\begin{equation*}
\left[ H,2Hz-AH\widetilde{F}_{1}-\frac{1}{3}b\widetilde{F}_{1}^{3}+\left( 
\frac{1}{4}bd+\frac{3}{64}\alpha d^{2}\right) H\widetilde{F}_{1}-\left( 
\frac{3}{8}\alpha d+2b\right) Hx\right] =0.
\end{equation*}
Since $2Hz-AH\widetilde{F}_{1}-\frac{1}{3}b\widetilde{F}_{1}^{3}+\left( 
\frac{1}{4}bd+\frac{3}{64}\alpha d^{2}\right) H\widetilde{F}_{1}-\left( 
\frac{3}{8}\alpha d+2b\right) Hx\in \Bbb{C}[x,y,z]_{3},$ we conclude that 
\begin{equation}
2Hz-AH\widetilde{F}_{1}-\frac{1}{3}b\widetilde{F}_{1}^{3}+\left( \frac{1}{4}%
bd+\frac{3}{64}\alpha d^{2}\right) H\widetilde{F}_{1}-\left( \frac{3}{8}%
\alpha d+2b\right) Hx=0.  \label{Row_lem_deg_5_H_10}
\end{equation}
Since $H|2Hz-AH\widetilde{F}_{1}+\left( \frac{1}{4}bd+\frac{3}{64}\alpha
d^{2}\right) H\widetilde{F}_{1}-\left( \frac{3}{8}\alpha d+2b\right) Hx,$ we
see that $b=0$ or $H|\widetilde{F}_{1}^{3}.$ But $H|\widetilde{F}_{1}^{3}$
means that $H$ is not squarefree. Thus $b=0.$ So (\ref{Row_lem_deg_5_H_10})
can be rewritten as 
\begin{equation*}
2Hz-AH\widetilde{F}_{1}+\frac{3}{64}\alpha d^{2}H\widetilde{F}_{1}-\frac{3}{8%
}\alpha dHx=0.
\end{equation*}
Thus 
\begin{eqnarray*}
z &=&\frac{1}{2}A\widetilde{F}_{1}-\frac{3}{128}\alpha d^{2}\widetilde{F}%
_{1}+\frac{3}{16}\alpha dx \\
&=&\left( \frac{3}{256}\alpha d^{2}+\frac{1}{4}c\right) \widetilde{F}_{1}-%
\frac{3}{128}\alpha d^{2}\widetilde{F}_{1}+\frac{3}{16}\alpha dx \\
&=&\left( -\frac{3}{256}\alpha d^{2}+\frac{1}{4}c\right) \widetilde{F}_{1}+%
\frac{3}{16}\alpha dx.
\end{eqnarray*}
The formulas for $G_{4},G_{3}$ and $G_{2}$ are obtained by substituting $%
b=0\,$in the formulas from Lemma \ref{Lem_deg_6_H}.
\end{proof}

Now we are in a position to prove

\begin{theorem}
There is no pair of polynomials $F,G$ of the form 
\begin{eqnarray*}
F &=&x+F_{2}+F_{3}+F_{4},\qquad F_{4}\neq 0, \\
G &=&z+G_{2}+\cdots +G_{6},\qquad G_{6}\neq 0,
\end{eqnarray*}
where $F_{4},G_{6}$ are given by the formulas of Lemma \ref{Lem_deg_10}(1),
such that $\deg \left[ F,G\right] <4.$
\end{theorem}

\begin{proof}
Assume that there exists such a pair. Then 
\begin{equation}
\left[ F_{3},z\right] +\left[ F_{2},G_{2}\right] +\left[ x,G_{3}\right] =0.
\label{Row_lem_deg_4_H_1}
\end{equation}
By Lemmas \ref{Lem_deg_8_H} and \ref{Lem_deg_5_H}, 
\begin{eqnarray}
\left[ F_{3},z\right] &=&\left[ H\widetilde{F}_{1},M\widetilde{F}_{1}+\frac{3%
}{16}\alpha dx\right]  \label{Row_lem_deg_4_H_2} \\
&=&H\left[ \widetilde{F}_{1},M\widetilde{F}_{1}+\frac{3}{16}\alpha dx\right]
+\widetilde{F}_{1}\left[ H,M\widetilde{F}_{1}+\frac{3}{16}\alpha dx\right] 
\notag \\
&=&\underset{--+--+--+--}{\frac{3}{16}\alpha dH\left[ \widetilde{F}%
_{1},x\right] }+M\widetilde{F}_{1}\left[ H,\widetilde{F}_{1}\right] +%
\underset{==========}{\frac{3}{16}\alpha d\widetilde{F}_{1}\left[ H,x\right] 
},  \notag
\end{eqnarray}
and by Lemmas \ref{Lem_deg_6_H} and \ref{Lem_deg_5_H}, 
\begin{eqnarray}
\left[ F_{2},G_{2}\right] &=&\left[ \frac{1}{4}\widetilde{F}_{1}^{2}+\frac{1%
}{4}dH,AH+\frac{3}{4}\alpha x\widetilde{F}_{1}\right]
\label{Row_lem_deg_4_H_3} \\
&=&\frac{1}{2}\widetilde{F}_{1}\left[ \widetilde{F}_{1},AH+\frac{3}{4}\alpha
x\widetilde{F}_{1}\right]  \notag \\
&&+\frac{1}{4}d\left[ H,AH+\frac{3}{4}\alpha x\widetilde{F}_{1}\right] 
\notag \\
&=&\frac{1}{2}A\widetilde{F}_{1}\left[ \widetilde{F}_{1},H\right] +\underset{%
---------}{\frac{3}{8}\alpha \widetilde{F}_{1}^{2}\left[ \widetilde{F}%
_{1},x\right] }  \notag \\
&&+\left[ H,\frac{3}{16}\alpha dx\widetilde{F}_{1}\right] ,  \notag
\end{eqnarray}
\begin{eqnarray}
\left[ x,G_{3}\right] &=&\left[ x,\frac{1}{8}\alpha \widetilde{F}_{1}^{3}+%
\frac{3}{16}\alpha dH\widetilde{F}_{1}+\frac{3}{2}\alpha Hx\right]
\label{Row_lem_deg_4_H_4} \\
&=&\underset{---------}{\frac{3}{8}\alpha \widetilde{F}_{1}^{2}\left[ x,%
\widetilde{F}_{1}\right] }+\underset{==========}{\frac{3}{16}\alpha d%
\widetilde{F}_{1}\left[ x,H\right] }  \notag \\
&&+\underset{--+--+--+--}{\frac{3}{16}\alpha dH\left[ x,\widetilde{F}%
_{1}\right] }+\frac{3}{2}\alpha x\left[ x,H\right] .  \notag
\end{eqnarray}
Notice that: 
\begin{equation}
\underset{------------------}{\frac{3}{8}\alpha \widetilde{F}_{1}^{2}\left[ 
\widetilde{F}_{1},x\right] +\frac{3}{8}\alpha \widetilde{F}_{1}^{2}\left[ x,%
\widetilde{F}_{1}\right] }=0,  \label{Row_lem_deg_4_H_5}
\end{equation}
\begin{equation}
\underset{====================}{\frac{3}{16}\alpha d\widetilde{F}_{1}\left[
H,x\right] +\frac{3}{16}\alpha d\widetilde{F}_{1}\left[ x,H\right] }=0,
\label{Row_lem_deg_4_H_6}
\end{equation}
\begin{equation}
\underset{--+--+--+--+--+--+--+--}{\frac{3}{16}\alpha dH\left[ \widetilde{F}%
_{1},x\right] +\frac{3}{16}\alpha dH\left[ x,\widetilde{F}_{1}\right] }=0.
\label{Row_lem_deg_4_H_7}
\end{equation}
By (\ref{Row_lem_deg_4_H_1})-(\ref{Row_lem_deg_4_H_7}) we have 
\begin{equation*}
\left[ H,\frac{1}{2}M\widetilde{F}_{1}^{2}-\frac{1}{4}A\widetilde{F}_{1}^{2}+%
\frac{3}{16}\alpha dx\widetilde{F}_{1}-\frac{3}{4}\alpha x^{2}\right] =0.
\end{equation*}
Since $\frac{1}{2}M\widetilde{F}_{1}^{2}-\frac{1}{4}A\widetilde{F}_{1}^{2}+%
\frac{3}{16}\alpha dx\widetilde{F}_{1}-\frac{3}{4}\alpha x^{2}\in \Bbb{C}%
[x,y,z]_{2},$ there is $e\in \Bbb{C}$ such that 
\begin{equation}
\frac{1}{2}M\widetilde{F}_{1}^{2}-\frac{1}{4}A\widetilde{F}_{1}^{2}+\frac{3}{%
16}\alpha dx\widetilde{F}_{1}-\frac{3}{4}\alpha x^{2}=eH.
\label{Row_lem_deg_4_H_8}
\end{equation}
We have (see Lemmas \ref{Lem_deg_6_H} and \ref{Lem_deg_5_H}) $\frac{1}{2}M-%
\frac{1}{4}A=-\frac{3}{256}\alpha d^{2}.$ Thus (\ref{Row_lem_deg_4_H_8}) can
be rewritten as 
\begin{equation*}
-\frac{3}{256}\alpha d^{2}\widetilde{F}_{1}^{2}+\frac{3}{16}\alpha dx%
\widetilde{F}_{1}-\frac{3}{4}\alpha x^{2}=eH
\end{equation*}
or equivalently 
\begin{equation*}
-\frac{3}{4}\alpha \left( \frac{1}{8}d\widetilde{F}_{1}-x\right) ^{2}=eH.
\end{equation*}
Since $H$ is squarefree and $\alpha \neq 0$, we see that $e=0$ and $\frac{1}{%
8}d\widetilde{F}_{1}-x=0.$ This means that $d\neq 0$ and $\widetilde{F}%
_{1}=8d^{-1}x.$ This contradicts $z=M\widetilde{F}_{1}+\frac{3}{16}\alpha
dx. $
\end{proof}

\section{The case of nonsquarefree $H$}

In this section we consider the situation of Lemma \ref{Lem_deg_10}(2).

\begin{lemma}
\label{Lem_deg_9} Let $\deg \left[ F,G\right] <9$ and let $\alpha $ and $h$
be as in Lemma \ref{Lem_deg_10}(2). Then there is $\beta \in \Bbb{C}$ such
that: 
\begin{equation*}
G_{5}=\frac{3}{2}\alpha h^{2}F_{3}+\beta h^{5}.
\end{equation*}
\end{lemma}

\begin{proof}
Since $\deg \left[ F,G\right] <9,$ it follows that 
\begin{equation*}
\left[ F_{4},G_{5}\right] +\left[ F_{3},G_{6}\right] =0.
\end{equation*}
By Lemma \ref{Lem_deg_10}(2), 
\begin{eqnarray*}
\left[ F_{4},G_{5}\right] +\left[ F_{3},G_{6}\right] &=&\left[
h^{4},G_{5}\right] +\left[ F_{3},\alpha h^{6}\right] \\
&=&4h^{3}\left[ h,G_{5}\right] +6\alpha h^{5}\left[ F_{3},h\right] \\
&=&\left[ h,4h^{3}G_{5}-6\alpha h^{5}F_{3}\right] .
\end{eqnarray*}
Thus $\left[ h,4h^{3}G_{5}-6\alpha h^{5}F_{3}\right] =0.$ Since $\deg h=1$
and $4h^{3}G_{5}-6\alpha h^{5}F_{3}\in \Bbb{C}[x,y,z]_{7},$ we conclude (by
Lemma \ref{Lemma_H_reduction}) that there exists $\Bbb{\beta \in C}$ such
that 
\begin{equation*}
4h^{3}G_{5}-6\alpha h^{5}F_{3}=4\beta h^{7}.
\end{equation*}
This gives the formula for $G_{5}.$
\end{proof}

\begin{lemma}
\label{Lem_deg_8} Let $\deg \left[ F,G\right] <8$ and let $\alpha ,\beta ,h$
be as in Lemma \ref{Lem_deg_9}. Then there is a homogeneous polynomial $%
\widetilde{F}_{2}$ of degree $2$ and $a\in \Bbb{C}$ such that: 
\begin{eqnarray*}
F_{3} &=&h\widetilde{F}_{2}, \\
G_{5} &=&\frac{3}{2}\alpha h^{3}\widetilde{F}_{2}+\beta h^{5}, \\
G_{4} &=&\frac{3}{8}\alpha \widetilde{F}_{2}^{2}+\frac{5}{4}\beta h^{2}%
\widetilde{F}_{2}+\frac{3}{2}\alpha h^{2}F_{2}+\frac{1}{4}ah^{4}.
\end{eqnarray*}
\end{lemma}

\begin{proof}
Since $\deg \left[ F,G\right] <8,$ it follows that 
\begin{equation*}
\left[ F_{4},G_{4}\right] +\left[ F_{3},G_{5}\right] +\left[
F_{2},G_{6}\right] =0.
\end{equation*}
By Lemmas \ref{Lem_deg_10}(2) and \ref{Lem_deg_9}, 
\begin{eqnarray*}
&&\left[ F_{4},G_{4}\right] +\left[ F_{3},G_{5}\right] +\left[
F_{2},G_{6}\right] \\
&=&\left[ h^{4},G_{4}\right] +\left[ F_{3},\frac{3}{2}\alpha
h^{2}F_{3}+\beta h^{5}\right] +\left[ F_{2},\alpha h^{6}\right] \\
&=&4h^{3}\left[ h,G_{4}\right] +3\alpha hF_{3}\left[ F_{3},h\right] +5\beta
h^{4}\left[ F_{3},h\right] +6\alpha h^{5}\left[ F_{2},h\right] \\
&=&\left[ h,4h^{3}G_{4}\right] +\left[ \frac{3}{2}\alpha hF_{3}^{2},h\right]
+\left[ 5\beta h^{4}F_{3},h\right] +\left[ 6\alpha h^{5}F_{2},h\right] \\
&=&\left[ h,4h^{3}G_{4}-\frac{3}{2}\alpha hF_{3}^{2}-5\beta
h^{4}F_{3}-6\alpha h^{5}F_{2}\right] .
\end{eqnarray*}
Thus $\left[ h,4h^{3}G_{4}-\frac{3}{2}\alpha hF_{3}^{2}-5\beta
h^{4}F_{3}-6\alpha h^{5}F_{2}\right] =0.$ By Lemma \ref{Lemma_H_reduction}
there is $a\in \Bbb{C}$ such that 
\begin{equation}
4h^{3}G_{4}-\frac{3}{2}\alpha hF_{3}^{2}-5\beta h^{4}F_{3}-6\alpha
h^{5}F_{2}=ah^{7}.  \label{Row_G4_1}
\end{equation}
Since $h|4h^{3}G_{4}-5\beta h^{4}F_{3}-6\alpha h^{5}F_{2}$ and $\alpha \neq
0,$ we see that $h^{3}|hF_{3}^{2}.\,$Thus $h|F_{3},$ and so there is a
homogeneous polynomial $\widetilde{F}_{2}$ of degree $2$ such that $F_{3}=h%
\widetilde{F}_{2}.$ By the last equality and (\ref{Row_G4_1}) we have 
\begin{equation*}
4h^{3}G_{4}-\frac{3}{2}\alpha h^{3}\widetilde{F}_{2}^{2}-5\beta h^{5}%
\widetilde{F}_{2}-6\alpha h^{5}F_{2}=ah^{7}.
\end{equation*}
This gives the formula for $G_{4}.$ The formula for $G_{5}$ is obtained by
substituting $F_{3}=h\widetilde{F}_{2}$ in the formula from Lemma \ref
{Lem_deg_9}
\end{proof}

\begin{lemma}
\label{Lem_deg_7} Let $\deg \left[ F,G\right] <7$ and let $\alpha ,\beta
,a,h,\widetilde{F}_{2}$ be as in Lemma \ref{Lem_deg_8}. Then there is a
homogeneous polynomial $\widetilde{F}_{1}$ of degree $1$ and $c\in \Bbb{C}$
such that 
\begin{eqnarray*}
F_{3} &=&h^{2}\widetilde{F}_{1}, \\
G_{5} &=&\frac{3}{2}\alpha h^{4}\widetilde{F}_{1}+\beta h^{5}, \\
G_{4} &=&\frac{3}{8}\alpha h^{2}\widetilde{F}_{1}^{2}+\frac{5}{4}\beta h^{3}%
\widetilde{F}_{1}+\frac{3}{2}\alpha h^{2}F_{2}+\frac{1}{4}ah^{4}, \\
G_{3} &=&\frac{5}{32}\beta h\widetilde{F}_{1}^{2}+\frac{1}{4}ah^{2}%
\widetilde{F}_{1}-\frac{1}{16}\alpha \widetilde{F}_{1}^{3}+\frac{5}{4}\beta
hF_{2}+\frac{3}{2}\alpha h^{2}x+\frac{3}{4}\alpha F_{2}\widetilde{F}_{1}+%
\frac{1}{4}ch^{3}.
\end{eqnarray*}
\end{lemma}

\begin{proof}
Since $\deg \left[ F,G\right] <7,$ it follows that 
\begin{equation}
\left[ F_{4},G_{3}\right] +\left[ F_{3},G_{4}\right] +\left[
F_{2},G_{5}\right] +\left[ x,G_{6}\right] =0.  \label{Row_lem_deg7_1}
\end{equation}
By Lemma \ref{Lem_deg_10}(2), 
\begin{equation}
\left[ F_{4},G_{3}\right] =\left[ h^{4},G_{3}\right] =4h^{3}\left[
h,G_{3}\right] =\left[ h,4h^{3}G_{3}\right]  \label{Row_lem_deg7_2}
\end{equation}
and 
\begin{equation}
\left[ x,G_{6}\right] =\left[ x,\alpha h^{6}\right] =6\alpha h^{5}\left[
x,h\right] =\left[ h,-6\alpha h^{5}x\right] .  \label{Row_lem_deg7_3}
\end{equation}
And by Lemma \ref{Lem_deg_8}, 
\begin{eqnarray}
&&\left[ F_{3},G_{4}\right]  \label{Row_lem_deg_7_3_a} \\
&=&\left[ h\widetilde{F}_{2},\frac{3}{8}\alpha \widetilde{F}_{2}^{2}+\frac{5%
}{4}\beta h^{2}\widetilde{F}_{2}+\frac{3}{2}\alpha h^{2}F_{2}+\frac{1}{4}%
ah^{4}\right]  \notag \\
&=&h\left[ \widetilde{F}_{2},\frac{3}{8}\alpha \widetilde{F}_{2}^{2}+\frac{5%
}{4}\beta h^{2}\widetilde{F}_{2}+\frac{3}{2}\alpha h^{2}F_{2}+\frac{1}{4}%
ah^{4}\right] +  \notag \\
&&\widetilde{F}_{2}\left[ h,\frac{3}{8}\alpha \widetilde{F}_{2}^{2}+\frac{5}{%
4}\beta h^{2}\widetilde{F}_{2}+\frac{3}{2}\alpha h^{2}F_{2}+\frac{1}{4}%
ah^{4}\right]  \notag \\
&=&\frac{5}{2}\beta h^{2}\widetilde{F}_{2}\left[ \widetilde{F}_{2},h\right]
+3\alpha h^{2}F_{2}\left[ \widetilde{F}_{2},h\right] +\frac{3}{2}\alpha
h^{3}\left[ \widetilde{F}_{2},F_{2}\right] +  \notag \\
&&ah^{4}\left[ \widetilde{F}_{2},h\right] +\frac{3}{4}\alpha \widetilde{F}%
_{2}^{2}\left[ h,\widetilde{F}_{2}\right] +\frac{5}{4}\beta h^{2}\widetilde{F%
}_{2}\left[ h,\widetilde{F}_{2}\right] +\frac{3}{2}\alpha h^{2}\widetilde{F}%
_{2}\left[ h,F_{2}\right]  \notag \\
&=&\underset{============}{-\left[ h,\frac{5}{4}\beta h^{2}\widetilde{F}%
_{2}^{2}\right] }+\underset{--+--+--+--}{3\alpha h^{2}F_{2}\left[ \widetilde{%
F}_{2},h\right] }+\underset{----------}{\frac{3}{2}\alpha h^{3}\left[ 
\widetilde{F}_{2},F_{2}\right] }  \notag \\
&&-\left[ h,ah^{4}\widetilde{F}_{2}\right] +\left[ h,\frac{1}{4}\alpha 
\widetilde{F}_{2}^{3}\right] +\underset{========}{\left[ h,\frac{5}{8}\beta
h^{2}\widetilde{F}_{2}^{2}\right] }+\underset{--+--+--+--}{\frac{3}{2}\alpha
h^{2}\widetilde{F}_{2}\left[ h,F_{2}\right] },  \notag
\end{eqnarray}
and 
\begin{eqnarray}
\left[ F_{2},G_{5}\right] &=&\left[ F_{2},\frac{3}{2}\alpha h^{3}\widetilde{F%
}_{2}+\beta h^{5}\right]  \label{Row_lem_deg_7_3_b} \\
&=&\frac{3}{2}\alpha h^{3}\left[ F_{2},\widetilde{F}_{2}\right] +\frac{9}{2}%
\alpha h^{2}\widetilde{F}_{2}\left[ F_{2},h\right] +5\beta h^{4}\left[
F_{2},h\right]  \notag \\
&=&\underset{---------}{\frac{3}{2}\alpha h^{3}\left[ F_{2},\widetilde{F}%
_{2}\right] }+\underset{--+--+--+--}{\frac{9}{2}\alpha h^{2}\widetilde{F}%
_{2}\left[ F_{2},h\right] }-\left[ h,5\beta h^{4}F_{2}\right] .  \notag
\end{eqnarray}
Notice that: 
\begin{equation}
\underset{-------------------}{\frac{3}{2}\alpha h^{3}\left[ \widetilde{F}%
_{2},F_{2}\right] +\frac{3}{2}\alpha h^{3}\left[ F_{2},\widetilde{F}%
_{2}\right] }=0,  \label{Row_lem_deg_7_3_c}
\end{equation}
\begin{equation}
\underset{==================}{-\left[ h,\frac{5}{4}\beta h^{2}\widetilde{F}%
_{2}^{2}\right] +\left[ h,\frac{5}{8}\beta h^{2}\widetilde{F}_{2}^{2}\right] 
}=\left[ h,-\frac{5}{8}\beta h^{2}\widetilde{F}_{2}^{2}\right] ,
\label{Row_lem_deg_7_3_d}
\end{equation}
and 
\begin{eqnarray}
&&\underset{--+--+--+--+--+--+--+--+--}{3\alpha h^{2}F_{2}\left[ \widetilde{F%
}_{2},h\right] +\frac{3}{2}\alpha h^{2}\widetilde{F}_{2}\left[
h,F_{2}\right] +\frac{9}{2}\alpha h^{2}\widetilde{F}_{2}\left[
F_{2},h\right] }  \label{Row_lem_deg_7_4} \\
&=&3\alpha h^{2}\left( F_{2}\left[ \widetilde{F}_{2},h\right] +\widetilde{F}%
_{2}\left[ F_{2},h\right] \right) =-\left[ h,3\alpha h^{2}F_{2}\widetilde{F}%
_{2}\right] .  \notag
\end{eqnarray}
By (\ref{Row_lem_deg7_1})-(\ref{Row_lem_deg_7_4}) we have 
\begin{equation*}
\left[ h,4h^{3}G_{3}-6\alpha h^{5}x-ah^{4}\widetilde{F}_{2}+\frac{1}{4}%
\alpha \widetilde{F}_{2}^{3}-5\beta h^{4}F_{2}-\frac{5}{8}\beta h^{2}%
\widetilde{F}_{2}^{2}-3\alpha h^{2}F_{2}\widetilde{F}_{2}\right] =0.
\end{equation*}
Thus there is $c\in \Bbb{C}$ such that 
\begin{gather}
4h^{3}G_{3}-6\alpha h^{5}x-ah^{4}\widetilde{F}_{2}+\frac{1}{4}\alpha 
\widetilde{F}_{2}^{3}  \label{Row_lem_deg7_5} \\
-5\beta h^{4}F_{2}-\frac{5}{8}\beta h^{2}\widetilde{F}_{2}^{2}-3\alpha
h^{2}F_{2}\widetilde{F}_{2}=ch^{6}.  \notag
\end{gather}
Since $h|4h^{3}G_{3}-6\alpha h^{5}x-ah^{4}\widetilde{F}_{2}-5\beta
h^{4}F_{2}-\frac{5}{8}\beta h^{2}\widetilde{F}_{2}^{2}-3\alpha h^{2}F_{2}%
\widetilde{F}_{2}$ and $\alpha \neq 0,$ we see that $h|\widetilde{F}_{2}.$
Thus there is a homogeneous polynomial $\widetilde{F}_{1}$ of degree $1$
such that 
\begin{equation}
\widetilde{F}_{2}=h\widetilde{F}_{1}.  \label{Row_lem_deg7_6}
\end{equation}
Now, Lemma \ref{Lem_deg_8} and (\ref{Row_lem_deg7_6}) give the formulas for $%
F_{3},$ $G_{5}$ and $G_{4},$ and (\ref{Row_lem_deg7_5})-(\ref{Row_lem_deg7_6}%
) give the formula for $G_{3}.$
\end{proof}

\begin{lemma}
\label{Lem_deg_6}Let $\deg \left[ F,G\right] <6,$ and let $\alpha ,\beta
,a,c,h,\widetilde{F}_{1}$ be as in Lemma \ref{Lem_deg_7}.\newline
(1) If $\beta =0,$ then there is a homogeneous polynomial $\widehat{F}_{1}$
of degree $1$ and $d\in \Bbb{C}$ such that 
\begin{eqnarray*}
F_{2} &=&\frac{1}{4}\left( \widetilde{F}_{1}^{2}+h\widehat{F}_{1}\right) , \\
G_{5} &=&\frac{3}{2}\alpha h^{4}\widetilde{F}_{1},\qquad G_{4}=\frac{3}{4}%
\alpha h^{2}\widetilde{F}_{1}^{2}+\frac{3}{8}\alpha h^{3}\widehat{F}_{1}+%
\frac{1}{4}ah^{4}, \\
G_{3} &=&\frac{1}{8}\alpha \widetilde{F}_{1}^{3}+\frac{1}{4}ah^{2}\widetilde{%
F}_{1}+\frac{3}{2}\alpha h^{2}x+\frac{3}{16}\alpha h\widetilde{F}_{1}%
\widehat{F}_{1}+\frac{1}{4}ch^{3}, \\
G_{2} &=&\frac{3}{16}ch\widetilde{F}_{1}+\frac{3}{128}\alpha \widehat{F}%
_{1}^{2}+\frac{1}{16}a\left( \widetilde{F}_{1}^{2}+h\widehat{F}_{1}\right) +%
\frac{3}{4}\alpha x\widetilde{F}_{1}+\frac{1}{4}dh^{2}.
\end{eqnarray*}
\newline
(2) If $\beta \neq 0,$ then there is a homogeneous polynomial $\overline{F}%
_{1}$ of degree $1$ and $b,d\in \Bbb{C}$ such that 
\begin{eqnarray*}
F_{2} &=&h\overline{F}_{1},\qquad \widetilde{F}_{1}=bh,\qquad F_{3}=bh^{3},
\\
G_{5} &=&\left( \frac{3}{2}\alpha b+\beta \right) h^{5},\qquad G_{4}=Eh^{4}+%
\frac{3}{2}\alpha h^{3}\overline{F}_{1}, \\
G_{3} &=&Kh^{3}+Lh^{2}\overline{F}_{1}+\frac{3}{2}\alpha h^{2}x, \\
G_{2} &=&Ahx+Bh^{2}+C\overline{F}_{1}^{2}+Dh\overline{F}_{1},
\end{eqnarray*}
where 
\begin{eqnarray*}
A &=&\frac{5}{4}\beta +\frac{3}{4}\alpha b,\qquad B=-\frac{5}{128}\beta
b^{3}+\frac{3}{16}cb+\frac{3}{128}\alpha b^{4}+\frac{1}{4}d, \\
C &=&\frac{3}{8}\alpha ,\qquad D=\frac{1}{4}a+\frac{5}{16}\beta b-\frac{3}{16%
}\alpha b^{2}, \\
E &=&\frac{3}{8}\alpha b^{2}+\frac{5}{4}\beta b+\frac{1}{4}a, \\
K &=&\frac{5}{32}\beta b^{2}+\frac{1}{4}ab-\frac{1}{16}\alpha b^{3}+\frac{1}{%
4}c,\qquad L=\frac{5}{4}\beta +\frac{3}{4}\alpha b.
\end{eqnarray*}
\end{lemma}

\begin{proof}
Since $\deg \left[ F,G\right] <6,$ we see that 
\begin{equation}
\left[ F_{4},G_{2}\right] +\left[ F_{3},G_{3}\right] +\left[
F_{2},G_{4}\right] +\left[ x,G_{5}\right] =0.  \label{Row_lem_deg6_1}
\end{equation}
By Lemma \ref{Lem_deg_10}(2), 
\begin{equation}
\left[ F_{4},G_{2}\right] =\left[ h^{4},G_{2}\right] =4h^{3}\left[
h,G_{2}\right] =\left[ h,4h^{3}G_{2}\right] ,  \label{Row_lem_deg6_2}
\end{equation}
and by Lemma \ref{Lem_deg_7}, 
\begin{eqnarray}
\left[ x,G_{5}\right] &=&\left[ x,\,\frac{3}{2}\alpha h^{4}\widetilde{F}%
_{1}+\beta h^{5}\right]  \label{Row_lem_deg6_2b} \\
&=&\underset{*************}{\frac{3}{2}\alpha h^{4}\left[ x,\widetilde{F}%
_{1}\right] }+\underset{==========}{6\alpha h^{3}\widetilde{F}_{1}\left[
x,h\right] }+5\beta h^{4}\left[ x,h\right] .  \notag
\end{eqnarray}
Also by Lemma \ref{Lem_deg_7}, 
\begin{eqnarray}
&&\left[ F_{3},G_{3}\right]  \label{Row_lem_deg6_4} \\
&=&\left[ h^{2}\widetilde{F}_{1},\frac{5}{32}\beta h\widetilde{F}_{1}^{2}+%
\frac{1}{4}ah^{2}\widetilde{F}_{1}-\frac{1}{16}\alpha \widetilde{F}_{1}^{3}+%
\frac{5}{4}\beta hF_{2}+\frac{3}{2}\alpha h^{2}x+\frac{3}{4}\alpha F_{2}%
\widetilde{F}_{1}+\frac{1}{4}ch^{3}\right]  \notag \\
&=&\left[ h^{2}\widetilde{F}_{1},\frac{5}{32}\beta h\widetilde{F}_{1}^{2}-%
\frac{1}{16}\alpha \widetilde{F}_{1}^{3}+\frac{5}{4}\beta hF_{2}+\frac{3}{2}%
\alpha h^{2}x+\frac{3}{4}\alpha F_{2}\widetilde{F}_{1}+\frac{1}{4}%
ch^{3}\right]  \notag \\
&=&h^{2}\left[ \widetilde{F}_{1},\frac{5}{32}\beta h\widetilde{F}_{1}^{2}-%
\frac{1}{16}\alpha \widetilde{F}_{1}^{3}+\frac{5}{4}\beta hF_{2}+\frac{3}{2}%
\alpha h^{2}x+\frac{3}{4}\alpha F_{2}\widetilde{F}_{1}+\frac{1}{4}%
ch^{3}\right]  \notag \\
&&+2h\widetilde{F}_{1}\left[ h,\frac{5}{32}\beta h\widetilde{F}_{1}^{2}-%
\frac{1}{16}\alpha \widetilde{F}_{1}^{3}+\frac{5}{4}\beta hF_{2}+\frac{3}{2}%
\alpha h^{2}x+\frac{3}{4}\alpha F_{2}\widetilde{F}_{1}+\frac{1}{4}%
ch^{3}\right]  \notag \\
&=&\underset{--\#--\#--\#--}{\frac{5}{32}\beta h^{2}\widetilde{F}%
_{1}^{2}\left[ \widetilde{F}_{1},h\right] }+\underset{----------}{\frac{5}{4}%
\beta h^{3}\left[ \widetilde{F}_{1},F_{2}\right] }+\underset{-*-*-*-*-*-*-}{%
\frac{5}{4}\beta h^{2}F_{2}\left[ \widetilde{F}_{1},h\right] }  \notag \\
&&+\underset{*************}{\frac{3}{2}\alpha h^{4}\left[ \widetilde{F}%
_{1},x\right] }+\underset{==========}{3\alpha h^{3}x\left[ \widetilde{F}%
_{1},h\right] }+\underset{+++++++++++}{\frac{3}{4}\alpha h^{2}\widetilde{F}%
_{1}\left[ \widetilde{F}_{1},F_{2}\right] }+\frac{3}{4}ch^{4}\left[ 
\widetilde{F}_{1},h\right]  \notag \\
&&+\underset{--\#--\#--\#--}{\frac{5}{8}\beta h^{2}\widetilde{F}%
_{1}^{2}\left[ h,\widetilde{F}_{1}\right] }-\frac{3}{8}\alpha h\widetilde{F}%
_{1}^{3}\left[ h,\widetilde{F}_{1}\right] +\underset{-*-*-*-*-*-*-}{\frac{5}{%
2}\beta h^{2}\widetilde{F}_{1}\left[ h,F_{2}\right] }  \notag \\
&&+\underset{==========}{3\alpha h^{3}\widetilde{F}_{1}\left[ h,x\right] }+%
\underset{\symbol{94}\symbol{94}\symbol{94}\symbol{94}\symbol{94}\symbol{94}%
\symbol{94}\symbol{94}\symbol{94}\symbol{94}\symbol{94}\symbol{94}\symbol{94}%
}{\frac{3}{2}\alpha h\widetilde{F}_{1}^{2}\left[ h,F_{2}\right] }+\underset{%
\symbol{94}\symbol{94}\symbol{94}\symbol{94}\symbol{94}\symbol{94}\symbol{94}%
\symbol{94}\symbol{94}\symbol{94}\symbol{94}\symbol{94}\symbol{94}}{\frac{3}{%
2}\alpha h\widetilde{F}_{1}F_{2}\left[ h,\widetilde{F}_{1}\right] }  \notag
\end{eqnarray}
and 
\begin{eqnarray}
\left[ F_{2},G_{4}\right] &=&\left[ F_{2},\frac{3}{8}\alpha h^{2}\widetilde{F%
}_{1}^{2}+\frac{5}{4}\beta h^{3}\widetilde{F}_{1}+\frac{3}{2}\alpha
h^{2}F_{2}+\frac{1}{4}ah^{4}\right]  \label{Row_lem_deg6_5} \\
&=&\underset{\symbol{94}\symbol{94}\symbol{94}\symbol{94}\symbol{94}\symbol{%
94}\symbol{94}\symbol{94}\symbol{94}\symbol{94}\symbol{94}\symbol{94}\symbol{%
94}}{\frac{3}{4}\alpha h\widetilde{F}_{1}^{2}\left[ F_{2},h\right] }+%
\underset{++++++++++++}{\frac{3}{4}\alpha h^{2}\widetilde{F}_{1}\left[ F_{2},%
\widetilde{F}_{1}\right] }+\underset{----------}{\frac{5}{4}\beta
h^{3}\left[ F_{2},\widetilde{F}_{1}\right] }  \notag \\
&&+\underset{-*-*-*-*-*-*-}{\frac{15}{4}\beta h^{2}\widetilde{F}_{1}\left[
F_{2},h\right] }+3\alpha hF_{2}\left[ F_{2},h\right] +ah^{3}\left[
F_{2},h\right] .  \notag
\end{eqnarray}
Notice that: 
\begin{eqnarray}
\underset{---------------------}{\frac{5}{4}\beta h^{3}\left[ \widetilde{F}%
_{1},F_{2}\right] +\frac{5}{4}\beta h^{3}\left[ F_{2},\widetilde{F}%
_{1}\right] } &=&0, \\
\underset{\ast *****************************}{\frac{3}{2}\alpha h^{4}\left[
x,\widetilde{F}_{1}\right] +\frac{3}{2}\alpha h^{4}\left[ \widetilde{F}%
_{1},x\right] } &=&0, \\
\underset{++++++++++++++++++++++++}{\frac{3}{4}\alpha h^{2}\widetilde{F}%
_{1}\left[ \widetilde{F}_{1},F_{2}\right] +\frac{3}{4}\alpha h^{2}\widetilde{%
F}_{1}\left[ F_{2},\widetilde{F}_{1}\right] } &=&0,
\end{eqnarray}
anf that: 
\begin{eqnarray}
&&\underset{--\#--\#--\#--}{\frac{5}{32}\beta h^{2}\widetilde{F}%
_{1}^{2}\left[ \widetilde{F}_{1},h\right] +\frac{5}{8}\beta h^{2}\widetilde{F%
}_{1}^{2}\left[ h,\widetilde{F}_{1}\right] } \\
&=&\frac{15}{32}\beta h^{2}\widetilde{F}_{1}^{2}\left[ h,\widetilde{F}%
_{1}\right] =\left[ h,\frac{5}{32}\beta h^{2}\widetilde{F}_{1}^{3}\right] , 
\notag
\end{eqnarray}
\begin{eqnarray}
&&\underset{-*-*-*-*-*-*-*-*-*-*-*-*-*-*-*-*-*-*-*-*-}{\frac{5}{4}\beta
h^{2}F_{2}\left[ \widetilde{F}_{1},h\right] +\frac{5}{2}\beta h^{2}%
\widetilde{F}_{1}\left[ h,F_{2}\right] +\frac{15}{4}\beta h^{2}\widetilde{F}%
_{1}\left[ F_{2},h\right] }  \label{Row_lem_deg6_7} \\
&=&\frac{5}{4}\beta h^{2}\left( F_{2}\left[ \widetilde{F}_{1},h\right] +%
\widetilde{F}_{1}\left[ F_{2},h\right] \right) =-\left[ h,\frac{5}{4}\beta
h^{2}\widetilde{F}_{1}F_{2}\right] ,  \notag
\end{eqnarray}
\begin{eqnarray}
&&\underset{=================================}{6\alpha h^{3}\widetilde{F}%
_{1}\left[ x,h\right] +3\alpha h^{3}x\left[ \widetilde{F}_{1},h\right]
+3\alpha h^{3}\widetilde{F}_{1}\left[ h,x\right] }  \label{Row_lem_deg6_8} \\
&=&3\alpha h^{3}\left( \widetilde{F}_{1}\left[ x,h\right] +x\left[ 
\widetilde{F}_{1},h\right] \right) =-\left[ h,3\alpha h^{3}x\widetilde{F}%
_{1}\right] ,  \notag
\end{eqnarray}
\begin{eqnarray}
&&\underset{\symbol{94}\symbol{94}\symbol{94}\symbol{94}\symbol{94}\symbol{94%
}\symbol{94}\symbol{94}\symbol{94}\symbol{94}\symbol{94}\symbol{94}\symbol{94%
}\symbol{94}\symbol{94}\symbol{94}\symbol{94}\symbol{94}\symbol{94}\symbol{94%
}\symbol{94}\symbol{94}\symbol{94}\symbol{94}\symbol{94}\symbol{94}\symbol{94%
}\symbol{94}\symbol{94}\symbol{94}\symbol{94}\symbol{94}\symbol{94}\symbol{94%
}\symbol{94}\symbol{94}\symbol{94}\symbol{94}\symbol{94}\symbol{94}\symbol{94%
}\symbol{94}\symbol{94}\symbol{94}\symbol{94}\symbol{94}\symbol{94}\symbol{94%
}\symbol{94}\symbol{94}\symbol{94}\symbol{94}}{\frac{3}{2}\alpha h\widetilde{%
F}_{1}^{2}\left[ h,F_{2}\right] +\frac{3}{2}\alpha h\widetilde{F}%
_{1}F_{2}\left[ h,\widetilde{F}_{1}\right] +\frac{3}{4}\alpha h\widetilde{F}%
_{1}^{2}\left[ F_{2},h\right] }  \label{Row_lem_deg6_9} \\
&=&\frac{3}{4}\alpha h\left( \widetilde{F}_{1}^{2}\left[ h,F_{2}\right] +2%
\widetilde{F}_{1}F_{2}\left[ h,\widetilde{F}_{1}\right] \right) =\left[ h,%
\frac{3}{4}\alpha h\widetilde{F}_{1}^{2}F_{2}\right] .  \notag
\end{eqnarray}
By (\ref{Row_lem_deg6_1})-(\ref{Row_lem_deg6_9}) we have 
\begin{eqnarray*}
&&\left[ h,4h^{3}G_{2}\right] +5\beta h^{4}\left[ x,h\right] +\frac{3}{4}%
ch^{4}\left[ \widetilde{F}_{1},h\right] -\frac{3}{8}\alpha h\widetilde{F}%
_{1}^{3}\left[ h,\widetilde{F}_{1}\right] +3\alpha hF_{2}\left[
F_{2},h\right] \\
&&+ah^{3}\left[ F_{2},h\right] +\left[ h,\frac{5}{32}\beta h^{2}\widetilde{F}%
_{1}^{3}\right] -\left[ h,\frac{5}{4}\beta h^{2}\widetilde{F}%
_{1}F_{2}\right] -\left[ h,3\alpha h^{3}x\widetilde{F}_{1}\right] +\left[ h,%
\frac{3}{4}\alpha h\widetilde{F}_{1}^{2}F_{2}\right] \\
&=&0
\end{eqnarray*}
or equivalently 
\begin{eqnarray}
&&\left[ h,4h^{3}G_{2}\right] -\left[ h,5\beta h^{4}x\right] -\left[ h,\frac{%
3}{4}ch^{4}\widetilde{F}_{1}\right]  \label{Row_lem_deg6_10} \\
&&-\left[ h,\frac{3}{32}\alpha h\widetilde{F}_{1}^{4}\right] -\left[ h,\frac{%
3}{2}\alpha hF_{2}^{2}\right] -\left[ h,ah^{3}F_{2}\right]  \notag \\
&&+\left[ h,\frac{5}{32}\beta h^{2}\widetilde{F}_{1}^{3}\right] -\left[ h,%
\frac{5}{4}\beta h^{2}\widetilde{F}_{1}F_{2}\right] -\left[ h,3\alpha h^{3}x%
\widetilde{F}_{1}\right] +\left[ h,\frac{3}{4}\alpha h\widetilde{F}%
_{1}^{2}F_{2}\right]  \notag \\
&=&0.  \notag
\end{eqnarray}
By the last equality and Lemma \ref{Lemma_H_reduction} there exists $d\in 
\Bbb{C}$ such that 
\begin{eqnarray}
&&4h^{3}G_{2}-5\beta h^{4}x-\frac{3}{4}ch^{4}\widetilde{F}_{1}-\frac{3}{32}%
\alpha h\widetilde{F}_{1}^{4}-\frac{3}{2}\alpha hF_{2}^{2}
\label{Row_lem_deg6_11} \\
&&-ah^{3}F_{2}+\frac{5}{32}\beta h^{2}\widetilde{F}_{1}^{3}-\frac{5}{4}\beta
h^{2}\widetilde{F}_{1}F_{2}-3\alpha h^{3}x\widetilde{F}_{1}+\frac{3}{4}%
\alpha h\widetilde{F}_{1}^{2}F_{2}  \notag \\
&=&dh^{5}.  \notag
\end{eqnarray}
Since $h^{2}|4h^{3}G_{2}-5\beta h^{4}x-\frac{3}{4}ch^{4}\widetilde{F}%
_{1}-ah^{3}F_{2}+\frac{5}{32}\beta h^{2}\widetilde{F}_{1}^{3}-\frac{5}{4}%
\beta h^{2}\widetilde{F}_{1}F_{2}-3\alpha h^{3}x\widetilde{F}_{1},$ we see
that 
\begin{equation*}
h^{2}|-\frac{3}{32}\alpha h\widetilde{F}_{1}^{4}+\frac{3}{4}\alpha h%
\widetilde{F}_{1}^{2}F_{2}-\frac{3}{2}\alpha hF_{2}^{2}=-\frac{3}{32}\alpha
h\left( \widetilde{F}_{1}^{2}-4F_{2}\right) ^{2}.
\end{equation*}
Thus $h|\widetilde{F}_{1}^{2}-4F_{2},$ and so there exists a homogeneous
polynomial $\widehat{F}_{1}$ of degree $1$ such that 
\begin{equation}
-\widehat{F}_{1}h=\widetilde{F}_{1}^{2}-4F_{2}.  \label{Row_lem_deg6_12}
\end{equation}
Then 
\begin{equation}
F_{2}=\frac{1}{4}\left( \widetilde{F}_{1}^{2}+h\widehat{F}_{1}\right) .
\label{Row_lem_deg6_13}
\end{equation}
Using (\ref{Row_lem_deg6_12})-(\ref{Row_lem_deg6_13}) we can rewrite (\ref
{Row_lem_deg6_11}) as 
\begin{eqnarray}
&&4h^{3}G_{2}-5\beta h^{4}x-\frac{3}{4}ch^{4}\widetilde{F}_{1}-\frac{3}{32}%
\alpha h^{3}\widehat{F}_{1}^{2}  \label{Row_lem_deg6_14} \\
&&-ah^{3}F_{2}+\frac{5}{32}\beta h^{2}\widetilde{F}_{1}^{3}-\frac{5}{4}\beta
h^{2}\widetilde{F}_{1}F_{2}-3\alpha h^{3}x\widetilde{F}_{1}  \notag \\
&=&dh^{5}.  \notag
\end{eqnarray}
Now, since $h^{3}|4h^{3}G_{2}-5\beta h^{4}x-\frac{3}{4}ch^{4}\widetilde{F}%
_{1}-\frac{3}{32}\alpha h^{3}\widehat{F}_{1}^{2}-ah^{3}F_{2}-3\alpha h^{3}x%
\widetilde{F}_{1},$ we conclude that 
\begin{equation*}
h^{3}|\frac{5}{32}\beta h^{2}\widetilde{F}_{1}^{3}-\frac{5}{4}\beta h^{2}%
\widetilde{F}_{1}F_{2}=\frac{5}{32}\beta h^{2}\widetilde{F}_{1}\left( 
\widetilde{F}_{1}^{2}-8F_{2}\right) .
\end{equation*}
Thus $\beta =0$ or $h|\widetilde{F}_{1}\left( \widetilde{F}%
_{1}^{2}-8F_{2}\right) .$

In the first case (i.e. $\beta =0$), by (\ref{Row_lem_deg6_12})-(\ref
{Row_lem_deg6_14}), 
\begin{equation*}
G_{2}=\frac{3}{16}ch\widetilde{F}_{1}+\frac{3}{128}\alpha \widehat{F}%
_{1}^{2}+\frac{1}{16}a\left( \widetilde{F}_{1}^{2}+h\widehat{F}_{1}\right) +%
\frac{3}{4}\alpha x\widetilde{F}_{1}+\frac{1}{4}dh^{2}.
\end{equation*}
The formulas for $G_{3},G_{4}$ and $G_{5}$ are obtained by substituting (\ref
{Row_lem_deg6_13}) and $\beta =0$ in the formulas from Lemma \ref{Lem_deg_7}.

In the second case (i.e. $h|\widetilde{F}_{1}\left( \widetilde{F}%
_{1}^{2}-8F_{2}\right) $ and $\beta \neq 0$) we obtain that $h|\widetilde{F}%
_{1},F_{2}.\,$Indeed, if $h|\widetilde{F}_{1},$ then $h|\frac{1}{4}\left( 
\widetilde{F}_{1}^{2}+h\widehat{F}_{1}\right) =F_{2}.$ And, if $h|\left( 
\widetilde{F}_{1}^{2}-8F_{2}\right) ,$ then $h|\left( \widetilde{F}%
_{1}^{2}-4F_{2}\right) -\left( \widetilde{F}_{1}^{2}-8F_{2}\right) =4F_{2},$
and so $h|\left( \widetilde{F}_{1}^{2}-4F_{2}\right) +4F_{2}=\widetilde{F}%
_{1}^{2}$ and $h|\widetilde{F}_{1}.$ Thus there exists a homogeneous
polynomial $\overline{F}_{1}$ of degree $1$ and $b\in \Bbb{C}$ such that 
\begin{equation}
F_{2}=\overline{F}_{1}h,\qquad \widetilde{F}_{1}=bh.  \label{Row_lem_deg6_15}
\end{equation}
By (\ref{Row_lem_deg6_15}) and (\ref{Row_lem_deg6_11}) we have 
\begin{eqnarray}
&&4h^{3}G_{2}-5\beta h^{4}x-\frac{3}{4}cbh^{5}-\frac{3}{32}\alpha b^{4}h^{5}-%
\frac{3}{2}\alpha h^{3}\overline{F}_{1}^{2}  \label{Row_lem_deg6_16} \\
&&-ah^{4}\overline{F}_{1}+\frac{5}{32}\beta b^{3}h^{5}-\frac{5}{4}\beta
bh^{4}\overline{F}_{1}-3\alpha bh^{4}x+\frac{3}{4}\alpha b^{2}h^{4}\overline{%
F}_{1}  \notag \\
&=&dh^{5}  \notag
\end{eqnarray}
or equivalently 
\begin{eqnarray}
G_{2} &=&\left( \frac{5}{4}\beta +\frac{3}{4}\alpha b\right) hx+\left( -%
\frac{5}{128}\beta b^{3}+\frac{3}{16}cb+\frac{3}{128}\alpha b^{4}+\frac{1}{4}%
d\right) h^{2}  \label{Row_lem_deg6_17} \\
&&+\frac{3}{8}\alpha \overline{F}_{1}^{2}+\left( \frac{1}{4}a+\frac{5}{16}%
\beta b-\frac{3}{16}\alpha b^{2}\right) h\overline{F}_{1}  \notag
\end{eqnarray}
The formulas for $G_{3},G_{4},G_{5}$ and $F_{3}$ are obtained by
substituting (\ref{Row_lem_deg6_15}) in the formulas from Lemma \ref
{Lem_deg_7}.
\end{proof}

Now we consider the situation of Lemma \ref{Lem_deg_6}(1), and we show that
if $\deg \left[ F,G\right] <5,$ then we do not need consider cases $\beta =0$
and $\beta \neq 0$ separately.

\begin{lemma}
\label{Lem_deg_5_beta_0_b} Let $\deg \left[ F,G\right] <5$ and let $\alpha
,\beta ,a,c,d,h,\widetilde{F}_{1},\widehat{F}_{1}$ be as in Lemma \ref
{Lem_deg_6}(1) (in particular $\beta =0$). Then there is $b\in \Bbb{C}$ such
that for $\overline{F}_{1}=\frac{1}{4}\left( b^{2}h+\widehat{F}_{1}\right) $
the formulas of Lemma \ref{Lem_deg_6}(2) holds true (of course with $\beta =0
$).
\end{lemma}

\begin{proof}
Since $\deg \left[ F,G\right] <5,$ we have 
\begin{equation}
\left[ F_{4},z\right] +\left[ F_{3},G_{2}\right] +\left[ F_{2},G_{3}\right]
+\left[ x,G_{4}\right] =0.  \label{Row_lem_deg5_beta0_1}
\end{equation}
By Lemma \ref{Lem_deg_10}(2), 
\begin{equation}
\left[ F_{4},z\right] =\left[ h^{4},z\right] =4h^{3}\left[ h,z\right]
=\left[ h,4h^{3}z\right] .  \label{Row_lem_deg5_beta0_2}
\end{equation}
By Lemma \ref{Lem_deg_7} and Lemma \ref{Lem_deg_6}(1), 
\begin{eqnarray}
&&\left[ F_{3},G_{2}\right]   \label{Row_lem_deg5_beta0_6} \\
&=&\left[ h^{2}\widetilde{F}_{1},\frac{3}{16}ch\widetilde{F}_{1}+\frac{3}{128%
}\alpha \widehat{F}_{1}^{2}+\frac{1}{16}a\widetilde{F}_{1}^{2}+\frac{1}{16}ah%
\widehat{F}_{1}+\frac{3}{4}\alpha x\widetilde{F}_{1}+\frac{1}{4}%
dh^{2}\right]   \notag \\
&=&h^{2}\left[ \widetilde{F}_{1},\frac{3}{16}ch\widetilde{F}_{1}+\frac{3}{128%
}\alpha \widehat{F}_{1}^{2}+\frac{1}{16}a\widetilde{F}_{1}^{2}+\frac{1}{16}ah%
\widehat{F}_{1}+\frac{3}{4}\alpha x\widetilde{F}_{1}+\frac{1}{4}%
dh^{2}\right]   \notag \\
&&+2h\widetilde{F}_{1}\left[ h,\frac{3}{16}ch\widetilde{F}_{1}+\frac{3}{128}%
\alpha \widehat{F}_{1}^{2}+\frac{1}{16}a\widetilde{F}_{1}^{2}+\frac{1}{16}ah%
\widehat{F}_{1}+\frac{3}{4}\alpha x\widetilde{F}_{1}+\frac{1}{4}%
dh^{2}\right]   \notag \\
&=&\frac{3}{16}ch^{2}\widetilde{F}_{1}\left[ \widetilde{F}_{1},h\right] +%
\underset{--**--**--**--}{\frac{3}{64}\alpha h^{2}\widehat{F}_{1}\left[ 
\widetilde{F}_{1},\widehat{F}_{1}\right] }+\underset{==*==*==*==}{\frac{1}{16%
}ah^{3}\left[ \widetilde{F}_{1},\widehat{F}_{1}\right] }+\underset{\symbol{94%
}\symbol{94}**\symbol{94}\symbol{94}**\symbol{94}\symbol{94}**\symbol{94}%
\symbol{94}}{\frac{1}{16}ah^{2}\widehat{F}_{1}\left[ \widetilde{F}%
_{1},h\right] }  \notag \\
&&+\underset{----------}{\frac{3}{4}\alpha h^{2}\widetilde{F}_{1}\left[ 
\widetilde{F}_{1},x\right] }+\frac{1}{2}dh^{3}\left[ \widetilde{F}%
_{1},h\right]   \notag \\
&&+\underset{---\#\#---\#\#---}{\frac{3}{8}ch^{2}\widetilde{F}_{1}\left[ h,%
\widetilde{F}_{1}\right] }+\underset{--+++--+++--}{\frac{3}{32}\alpha h%
\widetilde{F}_{1}\widehat{F}_{1}\left[ h,\widehat{F}_{1}\right] }+\underset{%
--\#--\#--\#--}{\frac{1}{4}ah\widetilde{F}_{1}^{2}\left[ h,\widetilde{F}%
_{1}\right] }  \notag \\
&&+\underset{++*++*++*++}{\frac{1}{8}ah^{2}\widetilde{F}_{1}\left[ h,%
\widehat{F}_{1}\right] }+\underset{*************}{\frac{3}{2}\alpha hx%
\widetilde{F}_{1}\left[ h,\widetilde{F}_{1}\right] }+\underset{-+--+--+-}{%
\frac{3}{2}\alpha h\widetilde{F}_{1}^{2}\left[ h,x\right] }.  \notag
\end{eqnarray}
By Lemma \ref{Lem_deg_6}(1), 
\begin{eqnarray}
\left[ x,G_{4}\right]  &=&\left[ x,\frac{3}{4}\alpha h^{2}\widetilde{F}%
_{1}^{2}+\frac{3}{8}\alpha h^{3}\widehat{F}_{1}+\frac{1}{4}ah^{4}\right] 
\label{Row_lem_deg5_beta0_3} \\
&=&\underset{----------}{\frac{3}{2}\alpha h^{2}\widetilde{F}_{1}\left[ x,%
\widetilde{F}_{1}\right] }+\underset{-+--+--+-}{\frac{3}{2}\alpha h%
\widetilde{F}_{1}^{2}\left[ x,h\right] }+\underset{+++++++++}{\frac{9}{8}%
\alpha h^{2}\widehat{F}_{1}\left[ x,h\right] }  \notag \\
&&+\underset{\symbol{94}\symbol{94}\symbol{94}\symbol{94}\symbol{94}\symbol{%
94}\symbol{94}\symbol{94}\symbol{94}\symbol{94}\symbol{94}\symbol{94}\symbol{%
94}}{\frac{3}{8}\alpha h^{3}\left[ x,\widehat{F}_{1}\right] }+ah^{3}\left[
x,h\right] ,  \notag
\end{eqnarray}
and 
\begin{eqnarray}
&&\left[ F_{2},G_{3}\right]   \label{Row_lem_deg5_beta0_5} \\
&=&\left[ \frac{1}{4}\left( \widetilde{F}_{1}^{2}+h\widehat{F}_{1}\right) ,%
\frac{1}{4}ah^{2}\widetilde{F}_{1}+\frac{1}{8}\alpha \widetilde{F}_{1}^{3}+%
\frac{3}{2}\alpha h^{2}x+\frac{3}{16}\alpha h\widehat{F}_{1}\widetilde{F}%
_{1}+\frac{1}{4}ch^{3}\right]   \notag \\
&=&\frac{1}{2}\widetilde{F}_{1}\left[ \widetilde{F}_{1},\frac{1}{4}ah^{2}%
\widetilde{F}_{1}+\frac{1}{8}\alpha \widetilde{F}_{1}^{3}+\frac{3}{2}\alpha
h^{2}x+\frac{3}{16}\alpha h\widehat{F}_{1}\widetilde{F}_{1}+\frac{1}{4}%
ch^{3}\right]   \notag \\
&&+\frac{1}{4}h\left[ \widehat{F}_{1},\frac{1}{4}ah^{2}\widetilde{F}_{1}+%
\frac{1}{8}\alpha \widetilde{F}_{1}^{3}+\frac{3}{2}\alpha h^{2}x+\frac{3}{16}%
\alpha h\widehat{F}_{1}\widetilde{F}_{1}+\frac{1}{4}ch^{3}\right]   \notag \\
&&+\frac{1}{4}\widehat{F}_{1}\left[ h,\frac{1}{4}ah^{2}\widetilde{F}_{1}+%
\frac{1}{8}\alpha \widetilde{F}_{1}^{3}+\frac{3}{2}\alpha h^{2}x+\frac{3}{16}%
\alpha h\widehat{F}_{1}\widetilde{F}_{1}+\frac{1}{4}ch^{3}\right]
\allowbreak   \notag \\
&=&\underset{--\#--\#--\#--}{\frac{1}{4}ah\widetilde{F}_{1}^{2}\left[ 
\widetilde{F}_{1},h\right] }+\underset{----------}{\frac{3}{4}\alpha h^{2}%
\widetilde{F}_{1}\left[ \widetilde{F}_{1},x\right] }+\underset{*************%
}{\frac{3}{2}\alpha hx\widetilde{F}_{1}\left[ \widetilde{F}_{1},h\right] } 
\notag \\
&&+\underset{==============}{\frac{3}{32}\alpha \widetilde{F}_{1}^{2}h\left[ 
\widetilde{F}_{1},\widehat{F}_{1}\right] }+\underset{--*--*--*--}{\frac{3}{32%
}\alpha \widetilde{F}_{1}^{2}\widehat{F}_{1}\left[ \widetilde{F}%
_{1},h\right] }+\underset{---\#\#---\#\#---}{\frac{3}{8}ch^{2}\widetilde{F}%
_{1}\left[ \widetilde{F}_{1},h\right] }  \notag \\
&&+\underset{==*==*==*==}{\frac{1}{16}ah^{3}\left[ \widehat{F}_{1},%
\widetilde{F}_{1}\right] }+\underset{++*++*++*++}{\frac{1}{8}ah^{2}%
\widetilde{F}_{1}\left[ \widehat{F}_{1},h\right] }+\underset{============}{%
\frac{3}{32}\alpha h\widetilde{F}_{1}^{2}\left[ \widehat{F}_{1},\widetilde{F}%
_{1}\right] }  \notag \\
&&+\underset{\symbol{94}\symbol{94}\symbol{94}\symbol{94}\symbol{94}\symbol{%
94}\symbol{94}\symbol{94}\symbol{94}\symbol{94}\symbol{94}\symbol{94}\symbol{%
94}}{\frac{3}{8}\alpha h^{3}\left[ \widehat{F}_{1},x\right] }+\underset{%
+++++++++}{\frac{3}{4}\alpha h^{2}x\left[ \widehat{F}_{1},h\right] }+%
\underset{--**--**--**--}{\frac{3}{64}\alpha h^{2}\widehat{F}_{1}\left[ 
\widehat{F}_{1},\widetilde{F}_{1}\right] }  \notag \\
&&+\underset{\symbol{94}\symbol{94}--\symbol{94}\symbol{94}--\symbol{94}%
\symbol{94}--\symbol{94}\symbol{94}}{\frac{3}{64}\alpha h\widehat{F}_{1}%
\widetilde{F}_{1}\left[ \widehat{F}_{1},h\right] }+\frac{3}{16}ch^{3}\left[ 
\widehat{F}_{1},h\right]   \notag \\
&&+\underset{\symbol{94}\symbol{94}**\symbol{94}\symbol{94}**\symbol{94}%
\symbol{94}**\symbol{94}\symbol{94}}{\frac{1}{16}ah^{2}\widehat{F}_{1}\left[
h,\widetilde{F}_{1}\right] }+\underset{--*--*--*--}{\frac{3}{32}\alpha 
\widehat{F}_{1}\widetilde{F}_{1}^{2}\left[ h,\widetilde{F}_{1}\right] }+%
\underset{+++++++++}{\frac{3}{8}\alpha h^{2}\widehat{F}_{1}\left[ h,x\right] 
}  \notag \\
&&+\underset{--+++--+++--}{\frac{3}{64}\alpha h\widehat{F}_{1}^{2}\left[ h,%
\widetilde{F}_{1}\right] }+\underset{\symbol{94}\symbol{94}--\symbol{94}%
\symbol{94}--\symbol{94}\symbol{94}--\symbol{94}\symbol{94}}{\frac{3}{64}%
\alpha h\widehat{F}_{1}\widetilde{F}_{1}\left[ h,\widehat{F}_{1}\right] }. 
\notag
\end{eqnarray}
Notice that: 
\begin{equation}
\underset{-------------------------------------}{\frac{3}{4}\alpha h^{2}%
\widetilde{F}_{1}\left[ \widetilde{F}_{1},x\right] +\frac{3}{2}\alpha h^{2}%
\widetilde{F}_{1}\left[ x,\widetilde{F}_{1}\right] +\frac{3}{4}\alpha h^{2}%
\widetilde{F}_{1}\left[ \widetilde{F}_{1},x\right] }=0,
\label{Row_lem_deg5_beta0_7}
\end{equation}
\begin{equation}
\underset{--+--+--+--+--+--+--}{\frac{3}{2}\alpha h\widetilde{F}%
_{1}^{2}\left[ h,x\right] +\frac{3}{2}\alpha h\widetilde{F}_{1}^{2}\left[
x,h\right] }=0,  \label{Row_lem_deg5_beta0_8}
\end{equation}
\begin{equation}
\underset{\symbol{94}\symbol{94}\symbol{94}\symbol{94}\symbol{94}\symbol{94}%
\symbol{94}\symbol{94}\symbol{94}\symbol{94}\symbol{94}\symbol{94}\symbol{94}%
\symbol{94}\symbol{94}\symbol{94}\symbol{94}\symbol{94}\symbol{94}\symbol{94}%
\symbol{94}\symbol{94}\symbol{94}\symbol{94}\symbol{94}}{\frac{3}{8}\alpha
h^{3}\left[ x,\widehat{F}_{1}\right] +\frac{3}{8}\alpha h^{3}\left[ \widehat{%
F}_{1},x\right] }=0,  \label{Row_lem_deg5_beta0_9}
\end{equation}
\begin{equation}
\underset{\ast *****************************}{\frac{3}{2}\alpha hx\widetilde{%
F}_{1}\left[ h,\widetilde{F}_{1}\right] +\frac{3}{2}\alpha hx\widetilde{F}%
_{1}\left[ \widetilde{F}_{1},h\right] }=0,  \label{Row_lem_deg5_beta0_10}
\end{equation}
\begin{equation}
\underset{=======================}{\frac{3}{32}\alpha \widetilde{F}%
_{1}^{2}h\left[ \widetilde{F}_{1},\widehat{F}_{1}\right] +\frac{3}{32}\alpha
h\widetilde{F}_{1}^{2}\left[ \widehat{F}_{1},\widetilde{F}_{1}\right] }=0,
\label{Row_lem_deg5_beta0_11}
\end{equation}
\begin{equation}
\underset{--*--*--*--*--*--*--}{\frac{3}{32}\alpha \widetilde{F}_{1}^{2}%
\widehat{F}_{1}\left[ \widetilde{F}_{1},h\right] +\frac{3}{32}\alpha 
\widehat{F}_{1}\widetilde{F}_{1}^{2}\left[ h,\widetilde{F}_{1}\right] }=0,
\label{Row_lem_deg5_beta0_12}
\end{equation}
\begin{equation}
\underset{==*==*==*==*==*==*==}{\frac{1}{16}ah^{3}\left[ \widetilde{F}_{1},%
\widehat{F}_{1}\right] +\frac{1}{16}ah^{3}\left[ \widehat{F}_{1},\widetilde{F%
}_{1}\right] }=0,  \label{Row_lem_deg5_beta0_13}
\end{equation}
\begin{equation}
\underset{++*++*++*++*++*++*++}{\frac{1}{8}ah^{2}\widetilde{F}_{1}\left[ h,%
\widehat{F}_{1}\right] +\frac{1}{8}ah^{2}\widetilde{F}_{1}\left[ \widehat{F}%
_{1},h\right] }=0,  \label{Row_lem_deg5_beta0_14}
\end{equation}
\begin{equation}
\underset{--**--**--**--**--**--**--}{\frac{3}{64}\alpha h^{2}\widehat{F}%
_{1}\left[ \widetilde{F}_{1},\widehat{F}_{1}\right] +\frac{3}{64}\alpha h^{2}%
\widehat{F}_{1}\left[ \widehat{F}_{1},\widetilde{F}_{1}\right] }=0,
\label{Row_lem_deg5_beta0_15}
\end{equation}
\begin{equation}
\underset{\symbol{94}\symbol{94}--\symbol{94}\symbol{94}--\symbol{94}\symbol{%
94}--\symbol{94}\symbol{94}--\symbol{94}\symbol{94}--\symbol{94}\symbol{94}--%
\symbol{94}\symbol{94}--\symbol{94}\symbol{94}--\symbol{94}\symbol{94}}{%
\frac{3}{64}\alpha h\widehat{F}_{1}\widetilde{F}_{1}\left[ \widehat{F}%
_{1},h\right] +\frac{3}{64}\alpha h\widehat{F}_{1}\widetilde{F}_{1}\left[ h,%
\widehat{F}_{1}\right] }=0,  \label{Row_lem_deg5_beta0_16}
\end{equation}
\begin{equation}
\underset{\symbol{94}\symbol{94}**\symbol{94}\symbol{94}**\symbol{94}\symbol{%
94}**\symbol{94}\symbol{94}**\symbol{94}\symbol{94}**\symbol{94}\symbol{94}**%
\symbol{94}\symbol{94}**\symbol{94}\symbol{94}**\symbol{94}\symbol{94}}{%
\frac{1}{16}ah^{2}\widehat{F}_{1}\left[ \widetilde{F}_{1},h\right] +\frac{1}{%
16}ah^{2}\widehat{F}_{1}\left[ h,\widetilde{F}_{1}\right] }=0,
\label{Row_lem_deg5_beta0_17}
\end{equation}
\begin{equation}
\underset{--\#--\#--\#--\#--\#--\#--}{\frac{1}{4}ah\widetilde{F}%
_{1}^{2}\left[ h,\widetilde{F}_{1}\right] +\frac{1}{4}ah\widetilde{F}%
_{1}^{2}\left[ \widetilde{F}_{1},h\right] }=0,
\label{Row_lem_deg5_beta0_17b}
\end{equation}
\begin{equation}
\underset{---\#\#---\#\#---\#\#---\#\#---}{\frac{3}{8}ch^{2}\widetilde{F}%
_{1}\left[ h,\widetilde{F}_{1}\right] +\frac{3}{8}ch^{2}\widetilde{F}%
_{1}\left[ \widetilde{F}_{1},h\right] }=0  \label{Row_lem_deg5_beta0_17c}
\end{equation}
and that: 
\begin{eqnarray}
&&\underset{++++++++++++++++++++++++++}{\frac{9}{8}\alpha h^{2}\widehat{F}%
_{1}\left[ x,h\right] +\frac{3}{4}\alpha h^{2}x\left[ \widehat{F}%
_{1},h\right] +\frac{3}{8}\alpha h^{2}\widehat{F}_{1}\left[ h,x\right] }
\label{Row_lem_deg5_beta0_18} \\
&=&\frac{3}{4}\alpha h^{2}\left( \widehat{F}_{1}\left[ x,h\right] +x\left[ 
\widehat{F}_{1},h\right] \right) =-\left[ h,\frac{3}{4}\alpha h^{2}\widehat{F%
}_{1}x\right] ,  \notag
\end{eqnarray}
\begin{eqnarray}
&&\underset{--+++--+++--+++--+++--}{\frac{3}{32}\alpha h\widetilde{F}_{1}%
\widehat{F}_{1}\left[ h,\widehat{F}_{1}\right] +\frac{3}{64}\alpha h\widehat{%
F}_{1}^{2}\left[ h,\widetilde{F}_{1}\right] }  \label{Row_lem_deg5_beta0_19}
\\
&=&\frac{3}{64}\alpha h\left( \widehat{F}_{1}^{2}\left[ h,\widetilde{F}%
_{1}\right] +2\widetilde{F}_{1}\widehat{F}_{1}\left[ h,\widehat{F}%
_{1}\right] \right) =\left[ h,\frac{3}{64}\alpha h\widetilde{F}_{1}\widehat{F%
}_{1}^{2}\right] .  \notag
\end{eqnarray}
By (\ref{Row_lem_deg5_beta0_1})-(\ref{Row_lem_deg5_beta0_19}), 
\begin{eqnarray*}
&&\left[ h,4h^{3}z\right] +\frac{3}{16}ch^{2}\widetilde{F}_{1}\left[ 
\widetilde{F}_{1},h\right] +\frac{1}{2}dh^{3}\left[ \widetilde{F}%
_{1},h\right] +ah^{3}\left[ x,h\right]  \\
&&+\frac{3}{16}ch^{3}\left[ \widehat{F}_{1},h\right] -\left[ h,\frac{3}{4}%
\alpha h^{2}x\widehat{F}_{1}\right] +\left[ h,\frac{3}{64}\alpha h\widetilde{%
F}_{1}\widehat{F}_{1}^{2}\right]  \\
&=&0
\end{eqnarray*}
or equivalently 
\begin{eqnarray}
&&\left[ h,4h^{3}z\right] -\left[ h,ah^{3}x\right] -\left[ h,\frac{3}{16}%
ch^{3}\widehat{F}_{1}\right] -\left[ h,\frac{3}{32}ch^{2}\widetilde{F}%
_{1}^{2}\right]   \label{Row_lem_deg5_beta0_20} \\
&&-\left[ h,\frac{1}{2}dh^{3}\widetilde{F}_{1}\right] -\left[ h,\frac{3}{4}%
\alpha h^{2}x\widehat{F}_{1}\right] +\left[ h,\frac{3}{64}\alpha h\widetilde{%
F}_{1}\widehat{F}_{1}^{2}\right]   \notag \\
&=&0.  \notag
\end{eqnarray}
By Lemma \ref{Lemma_H_reduction} there exists $g\in \Bbb{C}$ such that 
\begin{equation}
4h^{3}z-\frac{3}{32}ch^{2}\widetilde{F}_{1}^{2}-\frac{1}{2}dh^{3}\widetilde{F%
}_{1}-ah^{3}x-\frac{3}{16}ch^{3}\widehat{F}_{1}-\frac{3}{4}\alpha h^{2}x%
\widehat{F}_{1}+\frac{3}{64}\alpha h\widetilde{F}_{1}\widehat{F}%
_{1}^{2}=gh^{4}.  \label{Row_lem_deg5_beta0_21}
\end{equation}
Since $h^{2}|4h^{3}z-\frac{3}{32}ch^{2}\widetilde{F}_{1}^{2}-\frac{1}{2}%
dh^{3}\widetilde{F}_{1}-ah^{3}x-\frac{3}{16}ch^{3}\widehat{F}_{1}-\frac{3}{4}%
\alpha h^{2}x\widehat{F}_{1}$ and $\alpha \neq 0,$ we see that $h^{2}|h%
\widetilde{F}_{1}\widehat{F}_{1}^{2}.$ Then $h|\widetilde{F}_{1}$ or $h|%
\widehat{F}_{1}.$ Notice that in both cases $h|\widetilde{F}_{1}.$ Indeed,
if $h|\widehat{F}_{1}$ then 
\begin{equation}
h^{3}|4h^{3}z-\frac{1}{2}dh^{3}\widetilde{F}_{1}-ah^{3}x-\frac{3}{16}ch^{3}%
\widehat{F}_{1}-\frac{3}{4}\alpha h^{2}x\widehat{F}_{1}+\frac{3}{64}\alpha h%
\widetilde{F}_{1}\widehat{F}_{1}^{2}  \label{Row_lem_deg5_beta0_22}
\end{equation}
Then, by (\ref{Row_lem_deg5_beta0_21}) and (\ref{Row_lem_deg5_beta0_22}) $%
h^{3}|-\frac{3}{32}ch^{2}\widetilde{F}_{1}^{2}$ and so $h|\widetilde{F}_{1}.$
Thus there is $b\in \Bbb{C}$ such that 
\begin{equation}
\widetilde{F}_{1}=bh.  \label{Row_lem_deg5_beta0_23}
\end{equation}
Now, by (\ref{Row_lem_deg5_beta0_23}) and by Lemma \ref{Lem_deg_7}, $%
F_{3}=bh^{3}.$

And, by (\ref{Row_lem_deg5_beta0_23}) and by Lemma \ref{Lem_deg_6}(1), 
\begin{equation*}
F_{2}=\frac{1}{4}\left( b^{2}h+\widehat{F}_{1}\right) h.
\end{equation*}
Thus 
\begin{equation}
F_{2}=h\overline{F}_{1}\qquad \text{and}\qquad \widehat{F}_{1}=4\overline{F}%
_{1}-b^{2}h.  \label{Row_lem_deg5_beta0_25}
\end{equation}
Now, one can repeat the same arguments as in the last part of the proof of
Lemma \ref{Lem_deg_6} obtaining the same formulas as in Lemma \ref{Lem_deg_6}%
(2) (of course with $\beta =0$).
\end{proof}

Because of Lemma \ref{Lem_deg_5_beta_0_b}, it does not make a difference
whether $\beta =0$ or not. Thus in the following lemma we does not assume
anything about $\beta .$

\begin{lemma}
\label{Lem_deg_5_beta_n0}Let $\deg \left[ F,G\right] <5$ and let $\alpha
,a,b,c,d,h,A,B,C,DE,K,L,\overline{F}_{1}$ be as in Lemma \ref{Lem_deg_6}(2)
with arbitrary $\beta $ (see Lemma \ref{Lem_deg_5_beta_0_b}). Then either%
\newline
(1) there exist $R,S,M,N,\widetilde{K},\widetilde{L},\widetilde{B},%
\widetilde{D}\in \Bbb{C}$ such that 
\begin{eqnarray*}
x &=&R\overline{F}_{1}+Sh,\qquad z=M\overline{F}_{1}+Nh, \\
G_{3} &=&\widetilde{K}h^{3}+\widetilde{L}h^{2}\overline{F}_{1},\qquad G_{2}=%
\widetilde{B}h^{2}+C\overline{F}_{1}^{2}+\widetilde{D}h\overline{F}_{1},
\end{eqnarray*}
\newline
or\newline
(2) there exists $f\in \Bbb{C}$ such that 
\begin{eqnarray*}
\overline{F}_{1} &=&fh,\qquad F_{2}=fh^{2}, \\
G_{4} &=&\left( E+\frac{3}{2}\alpha f\right) h^{4},\qquad G_{3}=\left(
K+Lf\right) h^{3}+\frac{3}{2}\alpha h^{2}x, \\
G_{2} &=&Ahx+\left( B+Cf^{2}+Df\right) h^{2}
\end{eqnarray*}
and 
\begin{equation*}
h=\frac{1}{M}\left[ z-\left( E-\frac{3}{4}bA+\frac{3}{4}\alpha f\right)
x\right] ,
\end{equation*}
where 
\begin{equation*}
M=\left( \frac{3}{4}K-\frac{3}{4}bD\right) f+\frac{1}{4}e+\frac{1}{4}\left(
3bC-\frac{1}{2}L\right) f^{2}.
\end{equation*}
\end{lemma}

\begin{proof}
Since $\deg \left[ F,G\right] <5,$ we have 
\begin{equation}
\left[ F_{4},z\right] +\left[ F_{3},G_{2}\right] +\left[ F_{2},G_{3}\right]
+\left[ x,G_{4}\right] =0.  \label{Row_lem_deg5_beta_n0_1}
\end{equation}
By Lemma \ref{Lem_deg_10}(2), 
\begin{equation}
\left[ F_{4},z\right] =\left[ h^{4},z\right] =4h^{3}\left[ h,z\right]
=\left[ h,4h^{3}z\right]   \label{Row_lem_deg5_beta_n0_2}
\end{equation}
and by Lemmas \ref{Lem_deg_6}(2) and \ref{Lem_deg_5_beta_0_b}, 
\begin{equation}
\left[ x,G_{4}\right] =\left[ x,Eh^{4}+\frac{3}{2}\alpha h^{3}\overline{F}%
_{1}\right] =4Eh^{3}\left[ x,h\right] +\underset{--------}{\frac{9}{2}\alpha
h^{2}\overline{F}_{1}\left[ x,h\right] }+\underset{+++++++++}{\frac{3}{2}%
\alpha h^{3}\left[ x,\overline{F}_{1}\right] },
\label{Row_lem_deg5_beta_n0_3b}
\end{equation}
\begin{eqnarray}
\left[ F_{3},G_{2}\right]  &=&\left[ bh^{3},Ahx+Bh^{2}+C\overline{F}%
_{1}^{2}+Dh\overline{F}_{1}\right]   \label{Row_lem_deg5_beta_n0_4} \\
&=&3bh^{2}\left[ h,Ahx+Bh^{2}+C\overline{F}_{1}^{2}+Dh\overline{F}%
_{1}\right]   \notag \\
&=&\left[ h,3bAh^{3}x+3bCh^{2}\overline{F}_{1}^{2}+3bDh^{3}\overline{F}%
_{1}\right] ,  \notag
\end{eqnarray}
and
\begin{eqnarray}
&&\left[ F_{2},G_{3}\right]   \label{Row_lem_deg5_beta_n0_6} \\
&=&\left[ h\overline{F}_{1},Kh^{3}+Lh^{2}\overline{F}_{1}+\frac{3}{2}\alpha
h^{2}x\right]   \notag \\
&=&h\left[ \overline{F}_{1},Kh^{3}+Lh^{2}\overline{F}_{1}+\frac{3}{2}\alpha
h^{2}x\right] +\overline{F}_{1}\left[ h,Kh^{3}+Lh^{2}\overline{F}_{1}+\frac{3%
}{2}\alpha h^{2}x\right]   \notag \\
&=&3Kh^{3}\left[ \overline{F}_{1},h\right] +2Lh^{2}\overline{F}_{1}\left[ 
\overline{F}_{1},h\right] +\underset{---------}{3\alpha h^{2}x\left[ 
\overline{F}_{1},h\right] }+\underset{+++++++++}{\frac{3}{2}\alpha
h^{3}\left[ \overline{F}_{1},x\right] }  \notag \\
&&+Lh^{2}\overline{F}_{1}\left[ h,\overline{F}_{1}\right] +\underset{--------%
}{\frac{3}{2}\alpha h^{2}\overline{F}_{1}\left[ h,x\right] }.  \notag
\end{eqnarray}
Notice that: 
\begin{eqnarray}
&&\underset{------------------------------}{\frac{9}{2}\alpha h^{2}\overline{%
F}_{1}\left[ x,h\right] +3\alpha h^{2}x\left[ \overline{F}_{1},h\right] +%
\frac{3}{2}\alpha h^{2}\overline{F}_{1}\left[ h,x\right] }
\label{Row_lem_deg5_beta_n0_7} \\
&=&3\alpha h^{2}\left( \overline{F}_{1}\left[ x,h\right] +x\left[ \overline{F%
}_{1},h\right] \right) =3\alpha h^{2}\left[ x\overline{F}_{1},h\right]
=\left[ h,-3\alpha h^{2}x\overline{F}_{1}\right] ,  \notag
\end{eqnarray}
\begin{equation}
\underset{+++++++++++++++++++}{\frac{3}{2}\alpha h^{3}\left[ x,\overline{F}%
_{1}\right] +\frac{3}{2}\alpha h^{3}\left[ \overline{F}_{1},x\right] }=0.
\label{Row_lem_deg5_beta_n0_8}
\end{equation}
By (\ref{Row_lem_deg5_beta_n0_1})-(\ref{Row_lem_deg5_beta_n0_8}) 
\begin{eqnarray}
&&\left[ h,4h^{3}z\right] -\left[ h,4Eh^{3}x\right] +\left[
h,3bAh^{3}x+3bCh^{2}\overline{F}_{1}^{2}+3bDh^{3}\overline{F}_{1}\right] 
\label{Row_lem_deg5_beta_n0_9} \\
&&-\left[ h,3Kh^{3}\overline{F}_{1}\right] -\left[ h,Lh^{2}\overline{F}%
_{1}^{2}\right] +\left[ h,\frac{1}{2}Lh^{2}\overline{F}_{1}^{2}\right]
+\left[ h,-3\alpha h^{2}x\overline{F}_{1}\right]   \notag \\
&=&0  \notag
\end{eqnarray}
By the last equality and Lemma \ref{Lemma_H_reduction} there exists $e\in 
\Bbb{C}$ such that 
\begin{equation}
4h^{3}z-4Eh^{3}x+3bAh^{3}x+3bCh^{2}\overline{F}_{1}^{2}+3bDh^{3}\overline{F}%
_{1}-3Kh^{3}\overline{F}_{1}-\frac{1}{2}Lh^{2}\overline{F}_{1}^{2}-3\alpha
h^{2}x\overline{F}_{1}=eh^{4}.  \label{Row_lem_deg5_beta_n0_10}
\end{equation}
Since $h^{3}|4h^{3}z-4Eh^{3}x+3bAh^{3}x+3bDh^{3}\overline{F}_{1}-3Kh^{3}%
\overline{F}_{1},$ we see that $h^{3}|\left( 3bC-\frac{1}{2}L\right) h^{2}%
\overline{F}_{1}^{2}-3\alpha h^{2}x\overline{F}_{1}=h^{2}\overline{F}%
_{1}\left[ \left( 3bC-\frac{1}{2}L\right) \overline{F}_{1}-3\alpha x\right]
.\,$Thus $h|\left( 3bC-\frac{1}{2}L\right) \overline{F}_{1}-3\alpha x$ or $h|%
\overline{F}_{1}.$

In the first case (i.e. $h|\left( 3bC-\frac{1}{2}L\right) \overline{F}%
_{1}-3\alpha x$) there exists $\gamma \in \Bbb{C}$ such that 
\begin{equation}
\left( 3bC-\frac{1}{2}L\right) \overline{F}_{1}-3\alpha x=\gamma h
\label{Row_lem_deg5_beta_n0_10a}
\end{equation}
or equivalently 
\begin{eqnarray}
x &=&\frac{1}{3\alpha }\left[ \left( 3bC-\frac{1}{2}L\right) \overline{F}%
_{1}-\gamma h\right]   \label{Row_lem_deg5_beta_n0_10b} \\
&=&\left( \frac{1}{4}b-\frac{5}{24}\frac{\beta }{\alpha }\right) \overline{F}%
_{1}-\frac{\gamma }{3\alpha }h.  \notag
\end{eqnarray}
This gives the formula for $x.$ Using (\ref{Row_lem_deg5_beta_n0_10a}) we
can rewrite (\ref{Row_lem_deg5_beta_n0_10}) as 
\begin{equation*}
4h^{3}z-4Eh^{3}x+3bAh^{3}x+3bDh^{3}\overline{F}_{1}-3Kh^{3}\overline{F}%
_{1}+\gamma h^{3}\overline{F}_{1}=eh^{4}.
\end{equation*}
Thus 
\begin{equation}
z=\left( E-3bA\right) x+\left( \frac{3}{4}K-\frac{3}{4}bD-\frac{1}{4}\gamma
\right) \overline{F}_{1}+\frac{1}{4}eh.  \label{Row_lem_deg5_beta_n0_10c}
\end{equation}
Now, the formula for $z$ is obtained by substituting (\ref
{Row_lem_deg5_beta_n0_10b}) in (\ref{Row_lem_deg5_beta_n0_10c}), and the
formulas for $G_{3}$ and $G_{2}$ are obtained by substituting (\ref
{Row_lem_deg5_beta_n0_10b}) in the formulas of Lemma \ref{Lem_deg_6}(2) or 
\ref{Lem_deg_5_beta_0_b}.

In the second case (i.e. $h|\overline{F}_{1}$) there exists $f\in \Bbb{C}$
such that 
\begin{equation}
\overline{F}_{1}=fh.  \label{Row_lem_deg5_beta_n0_13}
\end{equation}
Then 
\begin{equation}
F_{2}=fh^{2}.  \label{Row_lem_deg5_beta_n0_14}
\end{equation}
Usimg (\ref{Row_lem_deg5_beta_n0_13}) we can rewrite (\ref
{Row_lem_deg5_beta_n0_10}) as
\begin{equation*}
4h^{3}z-4Eh^{3}x+3bAh^{3}x+3bf^{2}Ch^{4}+3bfDh^{4}-3fKh^{4}-\frac{1}{2}%
f^{2}Lh^{4}-3\alpha fh^{3}x=eh^{4}.
\end{equation*}
Thus
\begin{equation}
z=\left( E-\frac{3}{4}bA+\frac{3}{4}\alpha f\right) x+Mh.
\label{Row_lem_deg5_beta_n0_15}
\end{equation}
where 
\begin{equation*}
M=\left( \frac{3}{4}K-\frac{3}{4}bD\right) f+\frac{1}{4}e+\frac{1}{4}\left(
3bC-\frac{1}{2}L\right) f^{2}.
\end{equation*}
By (\ref{Row_lem_deg5_beta_n0_15}) $M\neq 0$ and $h=\frac{1}{M}\left[
z-\left( E-\frac{3}{4}bA+\frac{3}{4}\alpha f\right) x\right] .$ Substituting 
$\overline{F}_{1}=fh$ in the formulas of Lemma \ref{Lem_deg_6}(2) gives the
formulas for $G_{2},G_{3}$ and $G_{4}.$
\end{proof}

Now we are in a position to prove

\begin{theorem}
There is no pair of polynomials $F,G$ of the form 
\begin{eqnarray*}
F &=&x+F_{2}+F_{3}+F_{4},\qquad F_{4}\neq 0, \\
G &=&z+G_{2}+\cdots +G_{6},\qquad G_{6}\neq 0,
\end{eqnarray*}
where $F_{4},G_{6}$ are given by the formulas of Lemma \ref{Lem_deg_10}(2),
such that $\deg \left[ F,G\right] <4.$
\end{theorem}

\begin{proof}
Assume that there exists such a pair. Then 
\begin{equation}
\left[ F_{3},z\right] +\left[ F_{2},G_{2}\right] +\left[ x,G_{3}\right] =0.
\label{row_tw_ost_1}
\end{equation}
Assume that $F$ and $G$ satisfy Lemma  \ref{Lem_deg_5_beta_n0}(2). Then 
\begin{equation*}
\left[ F_{3},z\right] =\left[ bh^{3},z\right] =3bh^{2}\left[ h,z\right]
=\left[ h,3bh^{2}z\right] ,
\end{equation*}
\begin{eqnarray*}
\left[ F_{2},G_{2}\right]  &=&\left[ fh^{2},Ahx+\left( B+Cf^{2}+Df\right)
h^{2}\right]  \\
&=&2fh\left[ h,Ahx\right] =\left[ h,2fAh^{2}x\right] 
\end{eqnarray*}
and 
\begin{equation*}
\left[ x,G_{3}\right] =\left[ x,Ph^{3}+\frac{3}{2}\alpha h^{2}x\right] ,
\end{equation*}
where $P=K+Lf.$ Thus 
\begin{equation*}
\left[ x,G_{3}\right] =3Ph^{2}\left[ x,h\right] +3\alpha hx\left[ x,h\right]
=-\left[ h,3Ph^{2}x\right] -\left[ h,\frac{3}{2}\alpha hx^{2}\right] 
\end{equation*}
and so 
\begin{equation*}
\left[ h,3bh^{2}z+2fAh^{2}x-3Ph^{2}x-\frac{3}{2}\alpha hx^{2}\right] =0.
\end{equation*}
By Lemma \ref{Lemma_H_reduction} there exists $l\in \Bbb{C}$ such that 
\begin{equation*}
3bh^{2}z+2fh^{2}x-3Ph^{2}x-\frac{3}{2}\alpha hx^{2}=lh^{3}.
\end{equation*}
Since $h^{2}|3bh^{2}z+2fAh^{2}x-3Ph^{2}x$ and $\alpha \neq 0,$ we see that $%
h|x.$ But this means that $h=\mu x$ for some $\mu \in \Bbb{C}^{*}$ (remember
that $\alpha h^{6}\neq 0$). But this contradicts $h=\frac{1}{M}\left[
z-\left( E-\frac{3}{4}bA+\frac{3}{4}\alpha f\right) x\right] .$

Now, assume that $F$ and $G$ satisfy Lemma \ref{Lem_deg_5_beta_n0}(1). Then
\begin{equation}
\left[ F_{3},z\right] =\left[ bh^{3},M\overline{F}_{1}+Nh\right]
=3bMh^{2}\left[ h,\overline{F}_{1}\right] ,  \label{row_tw_ost_2}
\end{equation}
\begin{eqnarray}
\left[ F_{2},G_{2}\right]  &=&\left[ h\overline{F}_{1},\widetilde{B}h^{2}+C%
\overline{F}_{1}^{2}+\widetilde{D}h\overline{F}_{1}\right] =\left[ h%
\overline{F}_{1},\widetilde{B}h^{2}+C\overline{F}_{1}^{2}\right] 
\label{row_tw_ost_3} \\
&=&h\left[ \overline{F}_{1},\widetilde{B}h^{2}+C\overline{F}_{1}^{2}\right] +%
\overline{F}_{1}\left[ h,\widetilde{B}h^{2}+C\overline{F}_{1}^{2}\right]  
\notag \\
&=&2\widetilde{B}h^{2}\left[ \overline{F}_{1},h\right] +2C\overline{F}%
_{1}^{2}\left[ h,\overline{F}_{1}\right]   \notag
\end{eqnarray}
and 
\begin{eqnarray}
\left[ x,G_{3}\right]  &=&\left[ R\overline{F}_{1}+Sh,Kh^{3}+\widetilde{L}%
h^{2}\overline{F}_{1}\right]   \label{row_tw_ost_4} \\
&=&R\left[ \overline{F}_{1},Kh^{3}+\widetilde{L}h^{2}\overline{F}_{1}\right]
+S\left[ h,Kh^{3}+\widetilde{L}h^{2}\overline{F}_{1}\right]   \notag \\
&=&3KRh^{2}\left[ \overline{F}_{1},h\right] +2R\widetilde{L}h\overline{F}%
_{1}\left[ \overline{F}_{1},h\right] +S\widetilde{L}h^{2}\left[ h,\overline{F%
}_{1}\right] .  \notag
\end{eqnarray}
By (\ref{row_tw_ost_1})-(\ref{row_tw_ost_4}),
\begin{equation*}
\left[ \left( 3bM-2\widetilde{B}-3KR+S\widetilde{L}\right) h^{2}-2R%
\widetilde{L}h\overline{F}_{1}+2C\overline{F}_{1}^{2}\right] \cdot \left[ h,%
\overline{F}_{1}\right] =0.
\end{equation*}
By Lemma \ref{Lem_deg_5_beta_n0}(1), $h$ and $\overline{F}_{1}$ are
algebraically independent. Thus $\left[ h,\overline{F}_{1}\right] \neq 0.$
Since also $C=\frac{3}{8}\alpha \neq 0,$ we see that $\left( 3bM-2\widetilde{%
B}-3KR+S\widetilde{L}\right) h^{2}-2R\widetilde{L}h\overline{F}_{1}+2C%
\overline{F}_{1}^{2}\neq 0,$ a contradiction.
\end{proof}

\vspace{1cm}

\textsc{Marek Kara\'{s}\newline
Instytut Matematyki\newline
Uniwersytetu Jagiello\'{n}skiego\newline
ul. \L ojasiewicza 6}\newline
\textsc{30-348 Krak\'{o}w\newline
Poland\newline
} e-mail: Marek.Karas@im.uj.edu.pl

\end{document}